\documentclass[11pt,a4paper]{scrartcl}

\PassOptionsToPackage{mathfont=stix2}{mystyle-koma-enhanced}
\usepackage{mystyle-koma-enhanced}
\KOMAoptions{DIV=8}
\recalctypearea

\declaretheorem[style=mystyleplain,name=Theorem A,numbered=no]{mainthm}
\AddToHook{env/mainthm/begin}{\crefalias{section}{mainthm}}

\newcommand{\R}{\mathbb R}
\newcommand{\N}{\mathbb N}
\newcommand{\PP}{\mathbb P}
\newcommand{\EE}{\mathbb E}
\newcommand{\cC}{\mathcal C}
\newcommand{\cX}{\mathcal X}
\newcommand{\cN}{\mathcal N}
\newcommand{\whG}{\widehat G}
\newcommand{\Gp}{{}^pG}
\newcommand{\pwhG}{{}^p\widehat G}
\newcommand{\dd}{\,\mathrm d}
\newcommand{\rv}{\mathrm{RV}}
\newcommand{\1}{\mathbf 1}

\title{Spatial preferential attachment models: robustness revisited}
\author{
Peter Gracar\,\orcidlink{0000-0001-8340-8340}\thanks{University of Leeds,
Leeds, United Kingdom. Email:
\href{mailto:p.gracar@leeds.ac.uk}{\nolinkurl{p.gracar@leeds.ac.uk}}.}
\and
Christian M\"onch\,\orcidlink{0000-0002-6531-6482}\thanks{Independent
Researcher, Darmstadt, Germany. Email:
\href{mailto:cmoench25@gmail.com}{\nolinkurl{cmoench25@gmail.com}}.}
}
\date{}

\MSC{60K35, 05C80, 60G55, 82C43}
\Keywords{spatial preferential attachment, degree-based, robustness, percolation,
local weak convergence, scale-free networks}

\begin{document}
\maketitle

\begin{abstract}

We investigate robustness under percolation
in the degree-driven spatial preferential attachment model. Consider a profile
that is regularly varying with exponent \(-\delta\), where \(\delta>1\), and
an attachment rule with asymptotic slope \(\gamma\in(0,1)\). Jacob and
M\"orters (Ann.\ Probab.\ 45 (2017)) proved robustness when
\[
  \delta\left(\frac1\gamma-1\right)<1.
\]
We prove the converse in the strict regime: if
\[
  \delta\left(\frac1\gamma-1\right)>1,
\]
then sufficiently strong Bernoulli vertex thinning destroys the giant
component.
In particular, the strict robustness boundary coincides with that of the age-dependent
random connection model.

\end{abstract}

\section{Introduction}

Spatial preferential attachment provides a setting in which power-law degree
distributions coexist with geometric clustering. In the model introduced by
Jacob and M\"orters~\cite{JacobMorters2015}, vertices are born in space-time,
and a new vertex links to older vertices with probabilities depending on
distance and on the current degree of the older endpoint. An asymptotically
linear attachment rule with slope \(\gamma\in(0,1)\) produces a degree
distribution with power-law exponent \(\tau=1+1/\gamma\), while the spatial
profile controls the probability of long edges. The positive limiting average
clustering coefficient distinguishes these graphs from the locally tree-like
models for which branching-process approximations have been very effective.

A central question is robustness, that is, whether a giant component survives
independent vertex thinning for every positive retention probability. Jacob
and M\"orters~\cite[Theorem~1]{JacobMorters2017} showed that, for a profile
that is regularly varying with exponent \(-\delta\), where \(\delta>1\),
robustness holds when
\[
  \delta\left(\frac1\gamma-1\right)<1,
\]
and proved nonrobustness for \(\gamma<1/2\) in all dimensions and for
\(\gamma<(\delta-1)/\delta\) in dimension one. In
\cite[Remark~2]{JacobMorters2017} they conjectured nonrobustness throughout
the complementary strict regime \(\delta(1/\gamma-1)>1\). This regime includes
graphs with infinite degree variance, in which the geometry must prevent the
high-degree vertices from sustaining a robust giant. We prove this conjecture
in every dimension.

Subsequent work on the same degree-driven model established further global
and local properties. Hirsch and the second author~\cite{HirschMoench2020} proved that
typical distances in the robust giant are at most doubly logarithmic in the
number of vertices. The constant in their upper bound depends on the spatial
decay as well as on the attachment strength. They also obtained a large
deviation principle for the empirical neighbourhood structure.

The difficulty caused by the degree dependence also motivated the
age-dependent random connection model of the first author, Grauer,
L\"uchtrath and M\"orters~\cite{GracarGrauerLuechtrathMorters2019}. There the
random
attractiveness of a vertex born at time \(s\) is replaced, at a later time
\(t\), by a deterministic multiple of \((t/s)^\gamma\), the characteristic
scale of its degree growth. This preserves power-law degrees and geometric
clustering, while making the edges independent, conditionally on all vertex
locations and birth times. The resulting model retains substantial spatial
correlations after averaging over this environment, but removes the feedback
between an edge and subsequent attachment probabilities.

The age-dependent model belongs to the class of weight-dependent random
connection models, which also contains Boolean models, long-range percolation
and other spatial scale-free graphs; see
\cite{GracarLuechtrathMorters2021,GracarHeydenreichMoenchMorters2022}.
For the age-based specialisation, the first author, L\"uchtrath and
M\"orters~\cite{GracarLuechtrathMorters2021} identified the strict boundary
\(\delta(1/\gamma-1)=1\) between zero and positive critical bond-percolation
parameters. Their nonrobustness proof uses a geometric skeleton decomposition
to retain information lost by ordinary path counting. In that model,
conditional independence simplifies the probabilistic estimates. It does not
remove the geometric obstruction that distinguishes these models from their
nonspatial counterparts.

Later work established sharp chemical-distance asymptotics for the
age-dependent model and related geometric graphs
\cite{GracarGrauerMorters2022}, and recurrence and transience results for
random walks on weight-dependent networks
\cite{GracarHeydenreichMoenchMorters2022}. Jacob, Jahnel and
L\"uchtrath~\cite{JacobJahnelLuechtrath2026} give annulus-crossing criteria
which also apply to spatial graphs whose connections may depend on the
surrounding vertex configuration.

It is not automatic that the degree-driven and age-dependent models have the
same robustness boundary. Even without space, degree reinforcement can favour
the emergence of a giant component, compared with an age-based model with
independent edges and the same asymptotic single-edge probabilities. For
example, for the affine attachment rule \(f(k)=\gamma k+\beta\), with
\(0<\gamma<1/2\), the critical values of \(\beta\) for the existence of a
giant satisfy
\[
  \beta_{\mathrm c}^{\mathrm{degree}}
  =\frac{(1/2-\gamma)^2}{1-\gamma}
  <\frac{1-2\gamma}{4}
  =\beta_{\mathrm c}^{\mathrm{age}}.
\]
The first value is due to
Dereich and M\"orters~\cite[Proposition~1.3]{DereichMorters2013}; the second
is the threshold for the independent-edge graph with kernel
\(\beta(s\wedge t)^{-\gamma}(s\vee t)^{\gamma-1}\), recalled in
\cite[Section~1]{MortersSchleicher2026}. Thus the degree-driven model can
have a giant when its age-based counterpart does not, even though their
degree exponents and asymptotic single-edge probabilities agree.

Near-critical behaviour also depends on more than the degree exponent.
Eckhoff, M\"orters and
Ortgiese~\cite[Theorem~1.1]{EckhoffMortersOrtgiese2018} consider their
nonspatial preferential attachment model with \(\tau>3\) and a giant before
percolation. They show that, as the bond-retention parameter \(p\) decreases
to \(p_{\mathrm c}>0\), the giant-component density has an essential
singularity of the form \(\exp\{-(c+o(1))/\sqrt{p-p_{\mathrm c}}\}\), with
\(c>0\).
Their comparison with the polynomial
decay in rank-one models shows that the degree exponent alone does not
determine critical behaviour. The additional reinforcement in the spatial
model therefore left open the possibility of a robustness boundary below
the age-based value.

The smaller value \((\delta-1)/\delta\) in the one-dimensional bound of
Jacob and M\"orters also has a geometric meaning. For power-law
profiles with \(\delta>2\), the effective decay exponent of the age-dependent
kernel is \(\delta_{\mathrm{eff}}=2\) when
\(\gamma\le(\delta-1)/\delta\), and
\(\delta_{\mathrm{eff}}=1+\delta(1-\gamma)<2\) above this line
\cite[Proposition~1.3]{GracarLuechtrathMoench2025}. This exponent measures
the decay of connections between large vertex sets. The weak decay regime
\(\delta_{\mathrm{eff}}<2\) therefore corresponds to
\(\gamma>(\delta-1)/\delta\). Its consequences for cluster structure,
continuity of percolation and transience are studied in \cite{Moench2024}.
In related models, including the one-dimensional soft Boolean model, the
same line separates regimes with and without a supercritical phase
\cite[Theorem~1.4]{GracarLuechtrathMoench2025}. These questions differ from
robustness, but they show that the line \(\gamma=(\delta-1)/\delta\) is a
natural competing scale. In particular, it was open whether degree
reinforcement sustains robustness when \(\gamma>(\delta-1)/\delta\) and
\(\delta(1/\gamma-1)>1\).

Our result shows that the degree-driven model has the same strict robustness
boundary as the age-dependent model: reinforcement does not enlarge the
robust regime into the intervening parameter range. In particular,
sufficiently strong vertex thinning destroys the giant whenever
\[
  \delta\left(\frac1\gamma-1\right)>1,
  \qquad\text{equivalently,}\qquad
  \delta(\tau-2)>1.
\]
Together with \cite{JacobMorters2017}, this completes the phase diagram away
from the critical line; see \Cref{fig:phase-diagram}.

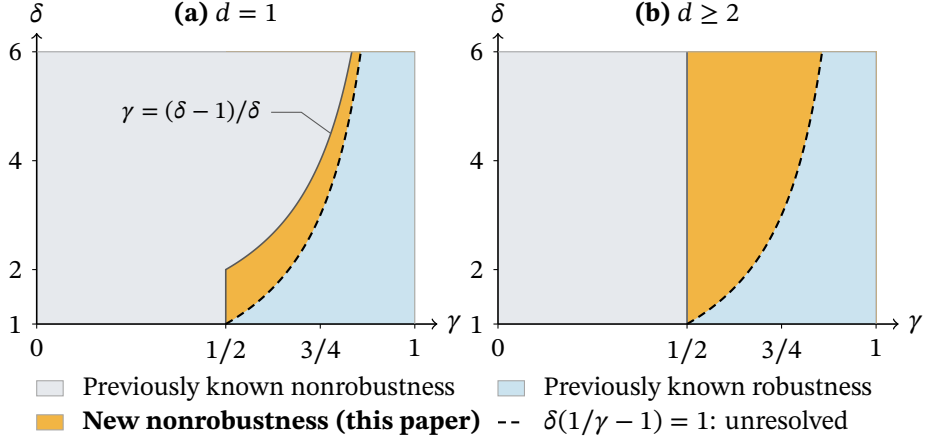
\begin{figure}[t]
\centering
\definecolor{phaseknown}{HTML}{E6E9ED}
\definecolor{phasenew}{HTML}{F2B544}
\definecolor{phaserobust}{HTML}{CBE3EF}
\begin{tikzpicture}[
  color=black,
  x=1cm,
  y=1cm,
  font=\small,
  axis/.style={->,semithick},
  critical line/.style={thick,densely dashed},
  previous boundary/.style={semithick,draw=black!65}
]
  \foreach \panel/\offset/\paneltitle in
    {1/0/{(a) $d=1$},2/6.1/{(b) $d\ge2$}} {
    \begin{scope}[shift={(\offset,0)}]
      \node[font=\small\bfseries] at (2.5,4.1) {\paneltitle};

      \fill[phasenew] (0,0) rectangle (5,3.6);
      \fill[phaserobust]
        (2.5,0) -- (5,0) -- (5,3.6) --
        plot[domain=6:1,samples=120,variable=\deltaval]
          ({5*\deltaval/(\deltaval+1)},{0.72*(\deltaval-1)}) -- cycle;

      \fill[phaseknown] (0,0) rectangle (2.5,3.6);
      \ifnum\panel=1
        \fill[phaseknown]
          (2.5,0.72) --
          plot[domain=2:6,samples=120,variable=\deltaval]
            ({5*(\deltaval-1)/\deltaval},{0.72*(\deltaval-1)}) --
          (2.5,3.6) -- cycle;
        \draw[previous boundary]
          (2.5,0) -- (2.5,0.72) --
          plot[domain=2:6,samples=120,variable=\deltaval]
            ({5*(\deltaval-1)/\deltaval},{0.72*(\deltaval-1)});
        \node[font=\footnotesize,anchor=east] (oldcurve) at (3.1,2.8)
          {$\gamma=(\delta-1)/\delta$};
        \draw[thin,black!65] (oldcurve.east) -- (3.5,2.8) -- (3.888889,2.52);
      \else
        \draw[previous boundary] (2.5,0) -- (2.5,3.6);
      \fi

      \draw[critical line]
        plot[domain=1:6,samples=120,variable=\deltaval]
          ({5*\deltaval/(\deltaval+1)},{0.72*(\deltaval-1)});
      \draw[axis] (0,0) -- (5.3,0) node[right] {$\gamma$};
      \draw[axis] (0,0) -- (0,3.85) node[above] {$\delta$};
      \draw[thin,black!35] (0,3.6) -- (5,3.6) -- (5,0);

      \foreach \x/\lab in {0/0,2.5/{1/2},3.75/{3/4},5/1} {
        \draw (\x,0) -- ++(0,-0.06) node[below] {$\lab$};
      }
      \foreach \y/\lab in {0/1,0.72/2,2.16/4,3.6/6} {
        \draw (0,\y) -- ++(-0.06,0) node[left] {$\lab$};
      }
    \end{scope}
  }

  \filldraw[fill=phaseknown,draw=black!35] (0,-0.75) rectangle (0.35,-0.95);
  \node[anchor=west] at (0.47,-0.85) {Previously known nonrobustness};
  \filldraw[fill=phaserobust,draw=black!35] (6.1,-0.75) rectangle (6.45,-0.95);
  \node[anchor=west] at (6.57,-0.85) {Previously known robustness};
  \filldraw[fill=phasenew,draw=black!35] (0,-1.2) rectangle (0.35,-1.4);
  \node[anchor=west,font=\small\bfseries] at (0.47,-1.3)
    {New nonrobustness (this paper)};
  \draw[critical line] (6.1,-1.3) -- (6.45,-1.3);
  \node[anchor=west] at (6.57,-1.3)
    {$\delta(1/\gamma-1)=1$: unresolved};
\end{tikzpicture}
\caption{The strict robustness phase diagram in the
\((\gamma,\delta)\)-plane, for \(\gamma\in(0,1)\) and \(\delta>1\).
Grey and blue indicate the nonrobust and robust regions established by
Jacob and M\"orters~\cite[Theorem~1]{JacobMorters2017}. Orange marks
the additional nonrobust region obtained in this paper, where
\(\delta(1/\gamma-1)>1\) and \(\gamma\) lies above the previous bounds
\(\max\{1/2,(\delta-1)/\delta\}\) in dimension one and \(1/2\) in higher
dimensions. The earlier nonrobustness bounds are strict; their boundary
points below the dashed critical line are also covered by our result.
The dashed line remains unresolved, and the vertical range is truncated for display.}
\label{fig:phase-diagram}
\end{figure}

The main technical issue is to control path probabilities while preserving
the spatial restrictions needed for nonrobustness. In the degree-driven
model, each new edge changes the future attractiveness of its older endpoint,
so successive attachments to that endpoint remain dependent even when all
vertex locations and birth times are fixed. Averaging over the Poisson
environment further couples the degree evolutions of different vertices,
since the same younger vertex may attach to several of them. One-edge
estimates on the scale
\((t/s)^\gamma\) do not by themselves justify the product bounds for long
paths that are directly available in the age-dependent model. Moreover,
integrating out the spatial variables too early loses the geometric
information needed to go beyond \(\gamma<1/2\).

To control these dependencies, we use the quick-path and Poisson-BK
decomposition of Jacob and M\"orters~\cite{JacobMorters2017}, together with
the skeleton contraction of \cite{GracarLuechtrathMorters2021}.
After domination by an affine attachment rule, degree-moment estimates with a
bounded number of forced arrivals control each short piece of a path. The
decomposition gives a bound that is exponential in the path length, with a
fixed constant per edge. This gives the static age-distance kernel bounds
needed for the geometric argument, without assuming independent edges.

To make the two decompositions compatible, we add deterministic short edges
after constructing the graph. An infinite retained component in this
augmented graph contains a geodesic ray. Quickening this ray and inserting
oldest common children produces a retained backbone with the required
separation between nonconsecutive vertices. We first integrate out the
inserted children and then contract local birth-time maxima of the retained
backbone. The remaining skeleton has birth times decreasing to one oldest
vertex and then increasing. A truncated first moment over these skeletons
is summable for sufficiently small retention probability. Finally, absence
of percolation in the infinite Palm limit rules out a giant in finite volume,
using the local convergence framework of
\cite{JacobMorters2015,JacobMorters2017}.

The structure of this paper is as follows. In \Cref{sec:model} we define the
model and state our main result. In \Cref{sec:spatial} we collect the spatial
integrals. In \Cref{sec:dynamic} we prove the degree-moment estimates and the
path estimate for augmented quick paths. In \Cref{sec:backbone} we construct
the separated backbone and reduce it to a skeleton. In \Cref{sec:skeleton} we
bound the first moment of skeletons and show that the infinite graph does not
percolate for small retention probability. In \Cref{sec:general} we extend
this to asymptotically linear attachment rules, and in \Cref{sec:finite} we
transfer it to finite graphs and complete the proof of Theorem~A.

\section{Model and result}\label{sec:model}

Fix \(d\ge1\). We write vertices as
\[
  \mathbf x=(x,s), \qquad x\in\R^d,\quad s\in(0,1],
\]
where smaller birth time means older. The infinite-volume vertex set
\(\cX\) is a unit-intensity Poisson process on \(\R^d\times(0,1]\).
For each ordered pair consisting of an older vertex
\(\mathbf x=(x,s)\) and a younger vertex \(\mathbf y=(y,t)\), \(s<t\),
let \(V_{\mathbf x,\mathbf y}\) be an independent uniform mark. The edge
\(\{\mathbf x,\mathbf y\}\) is present if and only if
\begin{equation}\label{eq:edge-rule}
  V_{\mathbf x,\mathbf y}
  \le
  \varphi\left(
    \frac{t\|x-y\|^d}{f(Z_{\mathbf x}(t-))}
  \right),
\end{equation}
where \(Z_{\mathbf x}(u)\) is the number of children of \(\mathbf x\) born
up to time \(u\). Edges are undirected after they are formed.

The profile \(\varphi:[0,\infty)\to(0,1)\) is nonincreasing and normalised by
\begin{equation}\label{eq:profile-normalization}
  \int_{\R^d}\varphi(\|z\|^d)\dd z=1.
\end{equation}
This is the standard \(d\)-dimensional convention. If \(v_d\) denotes the
volume of the Euclidean unit ball in \(\R^d\), then
\[
  \int_{\R^d}\varphi(\|z\|^d)\dd z
  =v_d\int_0^\infty\varphi(r)\dd r.
\]
Thus the assumption \(\varphi\in\rv_{-\delta}\) with \(\delta>1\) guarantees
the integrability required for this
normalisation.

The finite graph \(G^t\) is defined on the \(d\)-dimensional torus of volume
\(t\), with the same space-time construction and torus distance. We write
\(\Gp^t\) for the graph induced by independently retained vertices, each kept
with probability \(p\). Let \(C_1(\Gp^t)\) be a largest connected component,
with ties broken arbitrarily, and let \(C_2(\Gp^t)\) be a second-largest
component.

We say that the graph sequence is robust if, for every \(p\in(0,1]\), the
percolated sequence \((\Gp^t)_{t>0}\) has a giant component, meaning that
\(|C_1(\Gp^t)|/t\) stays bounded away from zero in probability. It is nonrobust if this fails for at least one \(p\in(0,1]\). In the
nonrobust phase, Theorem~A(ii) below in fact establishes the stronger
statement that, for all sufficiently small \(p\), the largest percolated
component is \(o(t)\) in
probability.

\begin{mainthm}\label{thm:main}
Assume that \(d\ge1\), that \(f:\N_0\to(0,\infty)\) is nondecreasing and
\[
  \frac{f(k)}{k}\longrightarrow\gamma\in(0,1),
\]
and that \(\varphi:[0,\infty)\to(0,1)\) is nonincreasing, satisfies
\eqref{eq:profile-normalization}, and belongs to \(\rv_{-\delta}\) for some
\(\delta>1\). Away from the critical line
\(\delta(1/\gamma-1)=1\), the strict robustness phase diagram is as follows.
\begin{enumerate}[label=(\roman*)]
\item If \(\delta(1/\gamma-1)<1\), then the sequence is
  robust. More
  precisely, for every \(p\in(0,1]\), there is
  \({}^p\theta>0\) such that
  \[
    |C_1(\Gp^t)|=({}^p\theta+o_{\PP}(1))t,
    \qquad
    |C_2(\Gp^t)|=o_{\PP}(t),
  \]
  where \(C_2\) is a second-largest component.
\item If \(\delta(1/\gamma-1)>1\), then the sequence is nonrobust. More
  precisely, there exists \(p_*>0\) such that, for every \(0<p<p_*\),
  \[
    \frac{|C_1(\Gp^t)|}{t}
    \xrightarrow[t\to\infty]{\PP}0 .
  \]
\end{enumerate}
\end{mainthm}

\begin{remark}[Interpretation of the product exponent]\label{rem:product-exponent}
Write \(\alpha:=1/\gamma-1=\tau-2\). If \(D\) has the limiting degree
distribution of the age-dependent model and \(D^\star\) has its size-biased
law, then the degree asymptotics in
\cite[Section~4]{GracarGrauerLuechtrathMorters2019} give
\[
  \PP\{D^\star>k\}
  =\frac{\EE[D\,\1\{D>k\}]}{\EE D}
  =k^{-\alpha+o(1)}.
\]
Thus \(\delta\alpha\) is the product of the connection-profile decay
exponent and the tail exponent of the size-biased degree distribution.

A scaling heuristic in the age-dependent model explains this product.
Vertices born at times of order \(k^{-1/\gamma}\) have degree scale \(k\)
and spatial density of order \(k^{-1/\gamma}\). Two such hubs at their
typical separation \(R\) therefore satisfy \(R^d\asymp k^{1/\gamma}\).
A ball of volume of order \(k\) around the first hub contains order \(k\)
candidate connectors born in \([1/2,1]\), each joining that hub with
probability bounded away from zero. For the edge to the second hub, the
normalised distance variable is of order
\(R^d/k\asymp k^\alpha\), and the connection probability is
\(k^{-\delta\alpha+o(1)}\). If \(N_k\) counts these two-step connections,
their expected number is therefore of order
\[
  \EE N_k=k^{1-\delta\alpha+o(1)}.
\]
For fixed positive retention probability, thinning changes only constants
in this estimate. The connector count thus grows with the hub scale in
the strict robust regime \(\delta\alpha<1\).

This interpretation in terms of the size-biased degree tail makes explicit
the balance underlying the robustness proofs of
Jacob and M\"orters~\cite[Lemma~11 and Proposition~13]{JacobMorters2017}
and Gracar, L\"uchtrath and M\"orters~\cite[Lemma~3.1]{GracarLuechtrathMorters2021}:
young connectors link a sequence of increasingly large hubs. The strict
inequality leaves room to increase the hub scale while keeping many
connectors.
\end{remark}

Item (i) is the robust phase of
\cite[Proposition~14]{JacobMorters2017}. The proof there is written in
dimension one in its final geometric step, but the authors state that no
argument is dimension-specific and that the same proof works in higher
dimensions mutatis mutandis. We prove item (ii). We first establish the
infinite-volume statement for an affine attachment rule
\[
  f(k)=\gamma k+\beta=\gamma(k+a),
  \qquad a=\beta/\gamma>0.
\]
Choose an exponent
\begin{equation}\label{eq:kappa}
  1<\eta<\delta,
  \qquad
  \kappa:=\eta(1-\gamma)-\gamma
  =\gamma\bigl[\eta(1/\gamma-1)-1\bigr]>0.
\end{equation}
Since \(\delta>1\), such an exponent exists if and only if
\(\delta(1/\gamma-1)>1\). Since \(\varphi\le1\), \(\varphi\in\rv_{-\delta}\)
and \(\eta<\delta\), there is a constant \(c_\eta<\infty\) such that
\begin{equation}\label{eq:profile-envelope}
  \varphi(r)\le c_\eta(1\wedge r^{-\eta}),\qquad r>0.
\end{equation}
This follows from Potter's bounds; see
\cite[Theorem~1.5.6]{BinghamGoldieTeugels1987}.
For \(\mathbf x=(x,s)\), \(\mathbf y=(y,t)\), set
\begin{equation}\label{eq:g-q-K}
  g(s,t):=(s\wedge t)^\gamma(s\vee t)^{1-\gamma},
  \qquad
  q(\mathbf x,\mathbf y):=g(s,t)\|x-y\|^d,
\end{equation}
and
\[
  K(\mathbf x,\mathbf y):=1\wedge q(\mathbf x,\mathbf y)^{-\eta}.
\]

\section{Spatial calculus}\label{sec:spatial}

The use of \(\|x-y\|^d\) makes the spatial integrals dimension-free up to a
constant.

\begin{lemma}\label{lem:one-kernel}
Let
\[
  J_{d,\eta}
  :=\int_{\R^d}(1\wedge\|z\|^{-d\eta})\dd z
  =v_d\frac{\eta}{\eta-1}<\infty .
\]
Then, for every \(A>0\),
\[
  \int_{\R^d}\left(1\wedge(A\|z\|^d)^{-\eta}\right)\dd z
  =J_{d,\eta}A^{-1}.
\]
\end{lemma}

\begin{proof}
Use the change of variables \(w=A^{1/d}z\). The value of the remaining
integral is
\[
  v_d+\int_{\|w\|>1}\|w\|^{-d\eta}\dd w
  =v_d+\frac{v_d}{\eta-1}
  =v_d\frac{\eta}{\eta-1}.
\]
\end{proof}

\begin{lemma}\label{lem:static-contraction}
Let \(\mathbf x=(x,t)\), \(\mathbf y=(y,s)\), with \(0<s<t\le1\), and suppose
\[
  q(\mathbf x,\mathbf y)
  =s^\gamma t^{1-\gamma}\|x-y\|^d\ge1.
\]
Then
\[
  \int_t^1\dd u\int_{\R^d}\dd z\,
  K(\mathbf x,(z,u))K((z,u),\mathbf y)
  \le
  \frac{C_{d,\eta}}{\kappa}K(\mathbf x,\mathbf y).
\]
\end{lemma}

\begin{proof}
Write \(L=\|x-y\|^d\). On the region
\(\|z-y\|\ge\|x-y\|/2\), bound the \(\mathbf y\)-leg by its power tail and
integrate the \(\mathbf x\)-leg with \Cref{lem:one-kernel}:
\[
\begin{aligned}
I_1
&\le
C L^{-\eta}s^{-\eta\gamma}
\int_t^1 u^{-\eta(1-\gamma)}
       \left(t^\gamma u^{1-\gamma}\right)^{-1}\dd u\\
&=
C L^{-\eta}s^{-\eta\gamma}t^{-\gamma}
\int_t^1u^{-1-\kappa}\dd u\\
&\le
\frac C\kappa
L^{-\eta}s^{-\eta\gamma}t^{-\eta(1-\gamma)}.
\end{aligned}
\]
The other region gives the same expression multiplied by
\((s/t)^{\gamma(\eta-1)}\le1\). The separation assumption puts the endpoint
kernel in its power-law regime, so the final expression is
\((C/\kappa)K(\mathbf x,\mathbf y)\).
\end{proof}

\begin{lemma}\label{lem:spatial-skeleton}
For fixed birth times \(t_0,\ldots,t_m\), fixed \(x_0=0\), and
\(\mathbf s_j=(x_j,t_j)\),
\[
  \int_{(\R^d)^m}
  \prod_{j=1}^mK(\mathbf s_{j-1},\mathbf s_j)
  \dd x_1\cdots\dd x_m
  =
  J_{d,\eta}^m\prod_{j=1}^m g(t_{j-1},t_j)^{-1}.
\]
\end{lemma}

\begin{proof}
Integrate successively in the spatial variables, using translation invariance
and \Cref{lem:one-kernel}.
\end{proof}

\section{Dynamic kernel estimates}\label{sec:dynamic}

Conditional on \(Z_{\mathbf x}(u)=k\), \eqref{eq:profile-normalization}
implies
\[
  \int_{\R^d}
  \varphi\left(\frac{u\|z-x\|^d}{f(k)}\right)\dd z
  =\frac{f(k)}u .
\]
Thus the logarithmic-time indegree process of a fixed vertex is the affine
pure-birth chain with jump rate \(k\mapsto\gamma(k+a)\).

\begin{lemma}\label{lem:degree-moments}
For every \(r>0\) and every integer \(M\ge0\), there is
\(C_{r,M}<\infty\) such that a vertex born at time \(s\), with at most
\(M\) additional forced inedges inserted at deterministic or stopping times,
satisfies
\[
  \EE\left[(a+Z_s^{[M]}(u))^r\right]
  \le
  C_{r,M}\left(\frac us\right)^{r\gamma},
  \qquad s\le u\le1.
\]
\end{lemma}

\begin{proof}[Proof of \Cref{lem:degree-moments}]
Write \(W(v)=a+Z_s^{[M]}(v)\). Conditional on
\(Z_{\mathbf x}(v-)=k\), the vertex gains a child at rate
\(f(k)/v=\gamma(k+a)/v\). Thus, in elapsed logarithmic time
\(\sigma=\log(v/s)\), the unforced process has jump rate \(\gamma W\). Forced
inedges add at most \(M\) unit upward jumps at deterministic or stopping times.

Let \(\bar W\) be a Yule process with per-particle rate \(\gamma\), started
from
\[
  N=\lceil a+M\rceil
\]
particles at elapsed time \(0\). Couple \(W\) and \(\bar W\) by matching each
spontaneous birth of \(W\) with a birth of \(\bar W\), letting \(\bar W\) make
additional births at rate \(\gamma(\bar W-W)\), and leaving the forced jumps of
\(W\) unmatched. The gap \(\bar W-W\) starts at \(N-a\ge M\), increases at the
additional births of \(\bar W\), and decreases by one at each of the at most
\(M\) forced jumps, so it remains nonnegative. Hence
\[
  W(u)\le \bar W(\log(u/s))
\]
for all \(u\), with a bound that depends on the number of interventions but
not on their times.

For a one-particle Yule process \(\bar W^{(1)}\),
\[
  \PP\{\bar W^{(1)}(\sigma)=k\}=\varrho(1-\varrho)^{k-1},\qquad k\ge1,
  \quad \varrho=e^{-\gamma\sigma}.
\]
For an integer \(\ell\ge1\), the bound
\[
  k^\ell\le k(k+1)\cdots(k+\ell-1)=\ell!\binom{k+\ell-1}{\ell}
\]
and
\[
  \sum_{k\ge1}\binom{k+\ell-1}{\ell}x^{k-1}=(1-x)^{-(\ell+1)}
\]
give
\[
  \EE\big[\bar W^{(1)}(\sigma)^\ell\big]
  =
  \varrho\sum_{k\ge1}k^\ell(1-\varrho)^{k-1}
  \le
  \ell!\,\varrho^{-\ell}
  =
  \ell!\,e^{\ell\gamma\sigma}.
\]
Since \(\bar W(\sigma)\) is a sum of \(N\) independent copies of
\(\bar W^{(1)}(\sigma)\), Minkowski's inequality yields
\[
  \EE[\bar W(\sigma)^\ell]\le N^\ell \ell!\,e^{\ell\gamma\sigma}.
\]
Fix \(r>0\) and put \(\ell=\lceil r\rceil\). By Jensen's inequality and the
coupling,
\[
  \EE\big[(a+Z_s^{[M]}(u))^r\big]
  \le
  \EE\big[\bar W(\log(u/s))^\ell\big]^{r/\ell}
  \le
  (N^\ell\ell!)^{r/\ell}\left(\frac us\right)^{r\gamma}.
\]
This proves the claim.
\end{proof}

We now prove the path estimate used in the backbone argument. It is a
pointwise version of \cite[Proposition~16]{JacobMorters2017}, and its proof
uses the quick paths, the splitting procedure and the BK inequality of
\cite[Section~6]{JacobMorters2017}. We give the details for two reasons. The
sequences to which we apply the estimate contain deterministic short
transitions, and we need a bound for every fixed choice of the spatial
positions, not only after spatial
integration.

A \emph{common child} of two vertices is a vertex that is younger than both
and adjacent to both in \(G^\infty\). Let \(\mathbf z_0,\ldots,\mathbf z_n\)
be distinct points of \(\R^d\times(0,1]\), written
\(\mathbf z_i=(z_i,u_i)\), with pairwise distinct birth times. The path
neighbours of \(\mathbf z_i\) are \(\mathbf z_{i-1}\) and \(\mathbf z_{i+1}\),
whenever these indices lie in \(\{0,\ldots,n\}\), and the pair
\(\{\mathbf z_{i-1},\mathbf z_i\}\) is the \(i\)th transition. A vertex
\(\mathbf z_i\) is \emph{age-regular} if it is older than at least one of its
path neighbours, and a \emph{local birth-time maximum} otherwise.

\begin{definition}\label{def:augmented-quick}
Let \(S\subseteq\{1,\ldots,n\}\) and \(R=\{1,\ldots,n\}\setminus S\). The
transitions with index in \(S\) are called \emph{short}, and those with index
in \(R\) are called \emph{required}. Under a law in which
\(\mathbf z_0,\ldots,\mathbf z_n\) are vertices, the sequence
\((\mathbf z_0,\ldots,\mathbf z_n)\) is an \emph{augmented quick path} with
short set \(S\) if the following three conditions hold.
\begin{enumerate}[label=(q\arabic*),start=0]
\item We have \(q(\mathbf z_{i-1},\mathbf z_i)\le1\) for every \(i\in S\),
  and \(\mathbf z_{i-1}\) and \(\mathbf z_i\) are adjacent in \(G^\infty\) for
  every \(i\in R\).
\item No age-regular vertex is adjacent to a younger vertex of the sequence
  other than its path neighbours.
\item Two age-regular vertices \(\mathbf z_i\) and \(\mathbf z_j\) with
  \(|i-j|\ge2\) have no common child, unless \(|i-j|=2\) and the vertex
  between them is a local birth-time maximum. In that case, no
  common child is older than the intermediate vertex.
\end{enumerate}
We write \(\mathcal Q_S\) for this event. A \emph{quick path} is an augmented
quick path with \(S=\varnothing\).
\end{definition}

For \(S=\varnothing\), conditions (q1) and (q2) are equivalent to the
properties (ii) and (iii) which define quick paths in
\cite[Section~6.1]{JacobMorters2017}. Indeed, if \(S=\varnothing\), then by
(q0) the intermediate vertex in (q2) is adjacent to its two path neighbours,
and it is younger than both. It is therefore a common child of them, and (q2)
says that it is their oldest common child. All
path estimates are stated for prescribed vertices with pairwise distinct birth
times. This loses no mass in the Palm integrations below, because the
birth-time coordinate is diffuse.

\begin{proposition}\label{prop:quick-path}
There is a constant \(C_{\mathrm{qp}}\ge1\), depending only on \(\gamma\),
\(a\), \(\eta\) and \(c_\eta\), such that the following holds. For every
\(n\ge1\), all distinct \(\mathbf z_0,\ldots,\mathbf z_n\in\R^d\times(0,1]\)
with pairwise distinct birth times, and every \(S\subseteq\{1,\ldots,n\}\),
\[
  \PP_{\mathbf z_0,\ldots,\mathbf z_n}(\mathcal Q_S)
  \le
  C_{\mathrm{qp}}^n
  \prod_{i=1}^nK(\mathbf z_{i-1},\mathbf z_i).
\]
\end{proposition}

Before giving the proof, we collect some facts about the Palm law
\(\PP_{\mathbf z_0,\ldots,\mathbf z_n}\). Under this law the vertex set is
\(X'\cup\{\mathbf z_0,\ldots,\mathbf z_n\}\), where \(X'\) is a Poisson point
process of unit intensity on \(\R^d\times(0,1]\), and all pairs of vertices
carry independent uniform marks. We call \(\mathbf z_0,\ldots,\mathbf z_n\)
the prescribed vertices, and for \(\mathbf w\in X'\) we write
\[
  m(\mathbf w)
  =\big(V_{\mathbf z_0,\mathbf w},\ldots,V_{\mathbf z_n,\mathbf w}\big)
  \in[0,1]^{n+1}.
\]
The edge rule \eqref{eq:edge-rule} decides whether an older vertex
\(\mathbf x\) and a younger vertex \(\mathbf y=(y,t)\) are adjacent using only
the mark \(V_{\mathbf x,\mathbf y}\), the positions and birth times of the two
vertices, and the indegree \(Z_{\mathbf x}(t-)\). Hence the indegree process
of a prescribed vertex \(\mathbf z_r\) is determined by the points of \(X'\)
that are younger than \(\mathbf z_r\), together with the \(r\)th coordinates
of their marks, and by the marks \(V_{\mathbf z_r,\mathbf z_j}\) of the
younger prescribed vertices \(\mathbf z_j\). In particular, the indegree
processes of different prescribed vertices depend on each other only through
the common point process \(X'\).

In the proof we compare graphs which are built from the same marks by the
rule \eqref{eq:edge-rule}, but in which some pairs of prescribed vertices are
deleted or inserted, or in which \(X'\) is replaced by a subset. A deleted pair
is never an edge, and an inserted pair is always an edge. Every other edge
between a prescribed vertex and a younger vertex is decided by
\eqref{eq:edge-rule}, using the indegree of the prescribed vertex in the graph
under consideration. In every graph of this type, each prescribed vertex has
finitely many children almost surely. For graphs built from all of \(X'\) this
follows from \Cref{lem:degree-moments}, since a prescribed vertex receives at
most \(n\) inedges from other prescribed vertices. The other graphs below are
built from finite subsets of \(X'\).

\begin{lemma}\label{lem:comparison}
Fix all marks, a prescribed vertex \(\mathbf z_r\) and \(T\in(u_r,1]\).
Consider two graphs of the type just described, in each of which
\(\mathbf z_r\) has finitely many children born before \(T\).
\begin{enumerate}[label=(\roman*)]
\item Suppose that every vertex born in \((u_r,T)\) which is a child of
  \(\mathbf z_r\) in at least one of the two graphs is a vertex of both
  graphs, and that it is either a child of \(\mathbf z_r\) in both graphs or
  has its edge to \(\mathbf z_r\) decided by \eqref{eq:edge-rule} in both
  graphs. Then the two graphs give \(\mathbf z_r\) the same children born
  before \(T\). In particular, the indegree processes of \(\mathbf z_r\) in the
  two graphs agree on \([u_r,T)\).
\item Suppose that the two graphs have the same deleted and the same inserted
  pairs, and that every vertex of the first graph is a vertex of the second.
  Then every child of \(\mathbf z_r\) in the first graph born before \(T\) is
  a child of \(\mathbf z_r\) in the second graph.
\end{enumerate}
\end{lemma}

\begin{proof}
For (i), suppose that the claim fails. Then the set of vertices born in
\((u_r,T)\) which are children of \(\mathbf z_r\) in exactly one of the two
graphs is finite and nonempty. Let \(\mathbf y=(y,t)\) be its oldest element.
By the choice of \(\mathbf y\), the two graphs give \(\mathbf z_r\) the same
children before time \(t\), and hence the same indegree
\(Z_{\mathbf z_r}(t-)\). By assumption, \(\mathbf y\) is either a child of
\(\mathbf z_r\) in both graphs, or its edge to \(\mathbf z_r\) is decided by
\eqref{eq:edge-rule} in both graphs, with the same mark and the same indegree.
In both cases \(\mathbf y\) is a child of \(\mathbf z_r\) in both graphs or in
neither. This contradicts the choice of \(\mathbf y\).

For (ii), suppose that the claim fails, and let \(\mathbf y=(y,t)\) be the
oldest child of \(\mathbf z_r\) in the first graph which is born before \(T\)
and is not a child of \(\mathbf z_r\) in the second graph. Every child of
\(\mathbf z_r\) in the first graph born before \(t\) is a child in the second
graph, so the indegree of \(\mathbf z_r\) at time \(t-\) is at least as large
in the second graph as in the first. The vertex \(\mathbf y\) is a vertex of
the second graph. Its edge to \(\mathbf z_r\) is not deleted, because
\(\mathbf y\) is a child of \(\mathbf z_r\) in the first graph. If the edge is
inserted, then \(\mathbf y\) is a child in both graphs. Otherwise the edge is
decided by \eqref{eq:edge-rule} in both graphs with the same mark. Since \(f\)
is nondecreasing and \(\varphi\) is nonincreasing, the right side of
\eqref{eq:edge-rule} is nondecreasing in the indegree. Hence \(\mathbf y\) is
a child of \(\mathbf z_r\) in the second graph, which is a contradiction.
\end{proof}

We write \(G^{(0)}\) for the graph obtained from \(G^\infty\) by deleting
every pair \(\{\mathbf z_i,\mathbf z_j\}\) with \(|i-j|\ge2\), and
\(\mathcal Q_S^{(0)}\) for the event of \Cref{def:augmented-quick} with
\(G^{(0)}\) in place of \(G^\infty\). This is the graph considered in
\cite[p.~1709]{JacobMorters2017}, where the indegree of \(\mathbf z_i\) is not
allowed to grow because of the vertices \(\mathbf z_{i'}\) with
\(|i'-i|\ge2\). There this is expressed by conditioning on
\(V(\mathbf z_i,\mathbf z_{i'})=0\). Under the edge rule
\eqref{eq:edge-rule}, which is \cite[Equation~(3)]{JacobMorters2017}, a mark
equal to zero produces an edge. The conditioning should therefore be on
\(V(\mathbf z_i,\mathbf z_{i'})=1\). Since \(\varphi<1\), a mark
equal to one produces no edge, so this conditioning has the same effect as
deleting these pairs.

\begin{lemma}\label{lem:suppression}
Under \(\PP_{\mathbf z_0,\ldots,\mathbf z_n}\), we have
\(\mathcal Q_S\subseteq\mathcal Q_S^{(0)}\) almost surely. Moreover, on
\(\mathcal Q_S\), every age-regular vertex has the same children in
\(G^\infty\) and in \(G^{(0)}\).
\end{lemma}

\begin{proof}
Assume that \(\mathcal Q_S\) occurs, and let \(\mathbf z_r\) be age-regular.
A vertex which is a child of \(\mathbf z_r\) in \(G^\infty\) or in
\(G^{(0)}\) is either a point of \(X'\) or a path neighbour of
\(\mathbf z_r\). Indeed, by (q1) no other prescribed vertex is a child of
\(\mathbf z_r\) in \(G^\infty\), and the corresponding pairs are deleted in
\(G^{(0)}\). The edges between \(\mathbf z_r\) and these vertices are decided
by \eqref{eq:edge-rule} in both graphs. \Cref{lem:comparison}(i), applied
with \(T=1\), shows that \(\mathbf z_r\) has the same children in the two
graphs. This proves the second statement.

We now check the conditions of \Cref{def:augmented-quick} in \(G^{(0)}\). The
condition on the short transitions in (q0) does not depend on the graph. The
older endpoint of a required transition is age-regular, so by the second
statement the transition is an edge of \(G^{(0)}\). Condition (q1) holds in
\(G^{(0)}\), because the pairs it concerns are deleted. A common child of two
age-regular vertices is a child of each of them. By the second statement, two
age-regular vertices have the same common children in \(G^\infty\) and in
\(G^{(0)}\), so (q2) holds in \(G^{(0)}\).
\end{proof}

We now recall the splitting procedure of
\cite[Section~6.2]{JacobMorters2017}, which depends only on the birth times.
We introduce a splitting at the index \(i\in\{0,\ldots,n\}\) if
\begin{equation}\label{eq:splitting}
  u_i>u_{i-1}\ \text{and}\ u_i>u_{i-2},
  \qquad\text{or}\qquad
  u_i>u_{i+1}\ \text{and}\ u_i>u_{i+2},
\end{equation}
where every inequality involving an index outside \(\{0,\ldots,n\}\) is
omitted. In particular, there is a splitting at \(0\) and at \(n\). Let
\(0=n_0<n_1<\cdots<n_k=n\) be the splitting indices. The \(j\)th part is
\((\mathbf z_{n_{j-1}},\ldots,\mathbf z_{n_j})\), and it contains the
transitions with indices \(n_{j-1}+1,\ldots,n_j\). Each transition belongs to
exactly one part. A vertex of the \(j\)th part is a \emph{local maximum of the
part} if it is younger than each of its path neighbours in the part. A vertex
\emph{contributes} to the \(j\)th part if it belongs to the part and is not a
local maximum of the part. Note that a vertex which contributes to some part
is age-regular.

\begin{lemma}\label{lem:parts}
The parts have the following properties.
\begin{enumerate}[label=(\roman*)]
\item Each part contains at most four transitions.
\item Each vertex contributes to at most one part.
\item The older endpoint of a transition contributes to the part containing
  the transition.
\item If \(\mathbf z_r\) contributes to the \(j\)th part, then every path
  neighbour of \(\mathbf z_r\) that is younger than \(\mathbf z_r\) belongs to
  the \(j\)th part, and so does the transition joining the two.
\end{enumerate}
\end{lemma}

\begin{proof}
For (i), suppose that \(n_j-n_{j-1}\ge5\). The four indices
\(n_{j-1}+1,\ldots,n_{j-1}+4\) are not splitting indices. Let \(i\) be the one
among them with the largest birth time. If \(i\ge n_{j-1}+3\), then \(i-1\) and
\(i-2\) are among the four indices, and the first alternative in
\eqref{eq:splitting} holds at \(i\). Otherwise \(i+1\) and \(i+2\) are among
the four indices, and the second alternative holds. In both cases \(i\) is a
splitting index, which is a contradiction.

For (ii), note that \(\mathbf z_i\) with \(n_{j-1}<i<n_j\) belongs only to the
\(j\)th part, and that \(\mathbf z_0\) and \(\mathbf z_n\) belong to one part
each. A vertex \(\mathbf z_{n_j}\) with \(0<j<k\) belongs to the \(j\)th and
to the \((j+1)\)th part. If it contributed to both, it would be older than both
of its path neighbours, and neither alternative in \eqref{eq:splitting} could
hold at \(n_j\).

For (iii), the older endpoint of a transition is older than the other
endpoint, which is one of its path neighbours in the same part. Hence it is
not a local maximum of this part.

For (iv), if \(n_{j-1}<r<n_j\), then both path neighbours of \(\mathbf z_r\)
and both transitions at \(\mathbf z_r\) belong to the \(j\)th part. Suppose
that \(r=n_j\) with \(j<k\). The only path neighbour of \(\mathbf z_r\) in the
\(j\)th part is \(\mathbf z_{r-1}\). Since \(\mathbf z_r\) contributes to this
part, we have \(u_{r-1}>u_r\). Hence the first alternative in
\eqref{eq:splitting} fails at \(r\), so the second one holds, and
\(u_{r+1}<u_r\). Thus \(\mathbf z_{r+1}\) is older than \(\mathbf z_r\). The
case \(r=n_{j-1}\) with \(j>1\) is symmetric. For \(r\in\{0,n\}\), the only
path neighbour of \(\mathbf z_r\) belongs to the part.
\end{proof}

For a vertex \(\mathbf z_r\) which contributes to some part, let \(T_r\) be
the largest birth time of a path neighbour of \(\mathbf z_r\) that is younger
than \(\mathbf z_r\). For \(1\le j\le k\), let \(\Gamma_j\) be the set of
points \(\mathbf w=(w,v)\in X'\) which are children in \(G^{(0)}\) of some
vertex \(\mathbf z_r\) contributing to the \(j\)th part, with
\(u_r<v<T_r\). Let \(A_j\) be the event that every required transition of the
\(j\)th part is an edge of \(G^{(0)}\). By \Cref{lem:parts}(iii), the older
endpoint \(\mathbf z_r\) of a required transition of the \(j\)th part
contributes to this part. The edge is decided at the birth time of the
younger endpoint, which is at most \(T_r\), using the indegree of
\(\mathbf z_r\) before that time. The next lemma is the analogue of
\cite[Lemma~18]{JacobMorters2017}.

\begin{lemma}\label{lem:disjoint-witnesses}
Almost surely on \(\mathcal Q_S^{(0)}\), the sets
\(\Gamma_1,\ldots,\Gamma_k\) are pairwise disjoint.
\end{lemma}

\begin{proof}
Suppose that \(\mathbf w=(w,v)\in\Gamma_j\cap\Gamma_{j'}\) with \(j\ne j'\).
Then \(\mathbf w\) is a child in \(G^{(0)}\) of a vertex \(\mathbf z_r\)
contributing to the \(j\)th part, with \(u_r<v<T_r\), and of a vertex
\(\mathbf z_{r'}\) contributing to the \(j'\)th part, with
\(u_{r'}<v<T_{r'}\). By \Cref{lem:parts}(ii), \(r\ne r'\), and we may assume
that \(r<r'\). Both vertices are age-regular, and \(\mathbf w\) is a common
child of them.

Suppose first that \(r'=r+1\). By \Cref{lem:parts}(iii), the older of the two
vertices contributes to the part containing the transition \(r+1\). By
\Cref{lem:parts}(ii), the younger one contributes to a different part, so it
belongs to two parts and its index is a splitting index. Assume first that
\(\mathbf z_r\) is the older one. Then \(\mathbf z_{r+1}\) contributes to the
part which starts at \(r+1\), so \(u_{r+1}<u_{r+2}\). The second alternative
in \eqref{eq:splitting} fails at \(r+1\), so the first one holds. This gives
\(u_{r+1}>u_r\) and, if \(r\ge1\), \(u_{r+1}>u_{r-1}\). Hence
\(T_r=u_{r+1}\), and \(v<T_r\) contradicts \(v>u_{r+1}\). If
\(\mathbf z_{r+1}\) is the older one, the same argument with the two sides
exchanged gives \(T_{r+1}=u_r\), and \(v<T_{r+1}\) contradicts \(v>u_r\).

Suppose next that \(r'=r+2\). By (q2), the vertex \(\mathbf z_{r+1}\) is a
local birth-time maximum and every common child of \(\mathbf z_r\)
and \(\mathbf z_{r+2}\) is born at time at least \(u_{r+1}\). Since
\(\mathbf w\in X'\), its birth time differs almost surely from
\(u_{r+1}\). Hence we have \(v>u_{r+1}\). If \(r+1\) were not a splitting
index, then
\(\mathbf z_r\), \(\mathbf z_{r+1}\) and \(\mathbf z_{r+2}\) would belong to
one part. Both \(\mathbf z_r\) and \(\mathbf z_{r+2}\) are older than
\(\mathbf z_{r+1}\), so both would contribute to this part, and
\Cref{lem:parts}(ii) would give \(j=j'\). Hence \(r+1\) is a splitting index.
If the first alternative in \eqref{eq:splitting} holds at \(r+1\), then
\(u_{r+1}>u_r\) and, if \(r\ge1\), \(u_{r+1}>u_{r-1}\). This gives
\(T_r=u_{r+1}<v\), which is a contradiction. If the second alternative holds,
we obtain \(T_{r+2}=u_{r+1}<v\) in the same way.

Finally, if \(r'\ge r+3\), or if \(r'=r+2\) and \(\mathbf z_{r+1}\) is not a
local birth-time maximum, then (q2) excludes common children of
\(\mathbf z_r\) and \(\mathbf z_{r'}\).
\end{proof}

Next we bound the probability of the event \(A_j\) for a single part. The
bound holds for every choice of the positions, and it uses the regular
variation of \(\varphi\) through \eqref{eq:profile-envelope}.

\begin{lemma}\label{lem:part-bound}
There is a constant \(C_1\ge1\), depending only on \(\gamma\), \(a\), \(\eta\)
and \(c_\eta\), such that the following holds. If \(R_j\) denotes the set of
required transitions of the \(j\)th part, then
\[
  \PP_{\mathbf z_0,\ldots,\mathbf z_n}(A_j)
  \le
  C_1^{|R_j|}\prod_{i\in R_j}K(\mathbf z_{i-1},\mathbf z_i).
\]
\end{lemma}

\begin{proof}
Write \(\PP=\PP_{\mathbf z_0,\ldots,\mathbf z_n}\). If \(R_j=\varnothing\),
there is nothing to prove. Otherwise, let \(G^*\) be the graph obtained from
\(G^{(0)}\) by inserting the required transitions of the \(j\)th part,
irrespective of their marks. For \(i\in R_j\), let \(o(i)\) and \(y(i)\) be
the indices of the older and of the younger endpoint of the \(i\)th
transition, and put \(s_i=u_{o(i)}\) and \(t_i=u_{y(i)}\). Let \(Z_i^*\) be the
indegree of \(\mathbf z_{o(i)}\) in \(G^*\) at time \(t_i-\), and
\[
  p_i
  =
  \varphi\left(
    \frac{t_i\|z_{o(i)}-z_{y(i)}\|^d}{f(Z_i^*)}
  \right).
\]

Assume that \(A_j\) occurs. Then every inserted pair is already an edge of
\(G^{(0)}\), so it is an edge of both graphs. Deleted pairs are edges of
neither graph. Every other edge between a prescribed vertex and a younger
vertex is decided by \eqref{eq:edge-rule} in both graphs. By
\Cref{lem:comparison}(i), applied with \(T=1\) to every prescribed vertex, the
two graphs give every prescribed vertex the same indegree process. The edge rule \eqref{eq:edge-rule} in \(G^{(0)}\) therefore
gives \(V_{\mathbf z_{i-1},\mathbf z_i}\le p_i\) for every \(i\in R_j\). The
graph \(G^*\) does not depend on the marks \(V_{\mathbf z_{i-1},\mathbf z_i}\),
\(i\in R_j\), and these marks are independent and uniform on \([0,1]\). Hence
\begin{equation}\label{eq:part-tower}
  \PP(A_j)
  \le
  \PP\left(V_{\mathbf z_{i-1},\mathbf z_i}\le p_i\text{ for all }i\in R_j\right)
  =
  \EE\Big[\prod_{i\in R_j}p_i\Big].
\end{equation}

Next we bound each \(p_i\). Put
\[
  H_i=q(\mathbf z_{o(i)},\mathbf z_{y(i)})
  =s_i^\gamma t_i^{1-\gamma}\|z_{o(i)}-z_{y(i)}\|^d,
  \qquad
  D_i=f(Z_i^*)\left(\frac{s_i}{t_i}\right)^\gamma .
\]
Then the argument of \(\varphi\) in \(p_i\) equals \(H_i/D_i\). If
\(H_i\le1\), then \(p_i\le1=K(\mathbf z_{i-1},\mathbf z_i)\). If \(H_i>1\),
then \eqref{eq:profile-envelope} gives
\(p_i\le c_\eta D_i^\eta H_i^{-\eta}=c_\eta D_i^\eta
K(\mathbf z_{i-1},\mathbf z_i)\). In both cases
\begin{equation}\label{eq:pi-bound}
  p_i\le K(\mathbf z_{i-1},\mathbf z_i)\big(1+c_\eta D_i^\eta\big).
\end{equation}

It remains to bound the moments of \(D_i\). Let \(r=o(i)\). In \(G^*\), the
vertex \(\mathbf z_r\) receives inedges from points of \(X'\) and from at most
two prescribed vertices, namely its path neighbours, since all other pairs
with prescribed vertices are deleted. An edge from \(\mathbf z_{r\pm1}\) is
either inserted or decided by its mark at the deterministic time
\(u_{r\pm1}\). It is therefore a unit jump at a stopping time, which equals
\(u_{r\pm1}\) if the edge is present and \(+\infty\) otherwise. Conditional on
the past, a point of \(X'\) born at time \(v\) is joined to \(\mathbf z_r\)
independently, and by \eqref{eq:profile-normalization} the points of \(X'\)
produce inedges of \(\mathbf z_r\) at rate
\(f(Z_{\mathbf z_r}(v-))/v\), as at the beginning of this section. Hence
\Cref{lem:degree-moments} applies with \(M=2\). Since
\(f(k)=\gamma(a+k)\) and the indegree is nondecreasing, we obtain, for every
\(\ell\in\{1,2,3,4\}\),
\[
  \EE\big[D_i^{\eta\ell}\big]
  \le
  \gamma^{\eta\ell}\left(\frac{s_i}{t_i}\right)^{\gamma\eta\ell}
  \EE\Big[\big(a+Z^{[2]}_{\mathbf z_r}(t_i)\big)^{\eta\ell}\Big]
  \le
  \gamma^{\eta\ell}C_{\eta\ell,2}.
\]
Put
\(C_2=\max_{1\le\ell\le4}\big(\gamma^{\eta\ell}C_{\eta\ell,2}\big)^{1/\ell}\)
and \(\ell=|R_j|\), so that \(\ell\le4\) by \Cref{lem:parts}(i). By
\eqref{eq:part-tower}, \eqref{eq:pi-bound}, H\"older's inequality with
\(\ell\) factors and Minkowski's inequality,
\[
  \PP(A_j)
  \le
  \prod_{i\in R_j}K(\mathbf z_{i-1},\mathbf z_i)
  \prod_{i\in R_j}
  \Big(1+c_\eta\,\EE\big[D_i^{\eta\ell}\big]^{1/\ell}\Big)
  \le
  (1+c_\eta C_2)^{\ell}
  \prod_{i\in R_j}K(\mathbf z_{i-1},\mathbf z_i).
\]
This proves the claim with \(C_1=1+c_\eta C_2\).
\end{proof}

We are now ready to prove \Cref{prop:quick-path}.

\begin{proof}[Proof of \Cref{prop:quick-path}]
Write \(\PP=\PP_{\mathbf z_0,\ldots,\mathbf z_n}\). By
\Cref{lem:suppression}, \(\PP(\mathcal Q_S)\le\PP(\mathcal Q_S^{(0)})\). The
condition on the short transitions in (q0) is deterministic, and
\(\1\{q(\mathbf z_{i-1},\mathbf z_i)\le1\}\le K(\mathbf z_{i-1},\mathbf z_i)\)
for \(i\in S\). Since \(C_1^{|R|}\le C_1^n\), it therefore suffices to show
that
\begin{equation}\label{eq:quick-reduced}
  \PP\big(\mathcal Q_S^{(0)}\big)
  \le
  C_1^{|R|}\prod_{i\in R}K(\mathbf z_{i-1},\mathbf z_i).
\end{equation}
Then the proposition holds with \(C_{\mathrm{qp}}=C_1\).

Let \(\Upsilon\) be the family of marks
\(V_{\mathbf z_{i-1},\mathbf z_i}\), \(1\le i\le n\). The marks of the deleted
pairs do not enter \(G^{(0)}\). Hence, conditional on \(\Upsilon\), the
events \(\mathcal Q_S^{(0)}\) and \(A_1,\ldots,A_k\) are determined by the
marked point process \(\{(\mathbf w,m(\mathbf w)):\mathbf w\in X'\}\). Under
\(\PP\) and conditional on \(\Upsilon\), this is a Poisson point process of
unit intensity on \(\R^d\times(0,1]\) with independent marks, uniformly
distributed on \([0,1]^{n+1}\).

In order to apply the BK inequality on a bounded domain, let
\(\Lambda_L=[-L,L]^d\times(0,1]\) for \(L>0\), and let \(A_j^L\) be the event
\(A_j\) for the graph built from the prescribed vertices and
\(X'\cap\Lambda_L\), with the same deleted pairs. A required transition is an
edge if and only if its younger endpoint is a child of its older endpoint. By
\Cref{lem:comparison}(ii), \(A_j^L\) is therefore an increasing event of the
marked point process on \(\Lambda_L\), and \(A_j^L\subseteq A_j\).

Assume that \(\mathcal Q_S^{(0)}\) occurs. The sets \(\Gamma_j\) are finite
almost surely, and by \Cref{lem:disjoint-witnesses} they are pairwise
disjoint. Let \(L\) be so large that \(\Lambda_L\) contains
\(\Gamma_1\cup\cdots\cup\Gamma_k\). Fix \(j\), and consider the graph built
from the prescribed vertices and \(\Gamma_j\) only. Let \(\mathbf z_r\)
contribute to the \(j\)th part. Every child of \(\mathbf z_r\) in
\(G^{(0)}\) born in \((u_r,T_r)\) is either a point of \(\Gamma_j\) or a path
neighbour of \(\mathbf z_r\). The same holds in the graph built from
\(\Gamma_j\), and all vertices of this graph are vertices of \(G^{(0)}\). The
edges between \(\mathbf z_r\) and these vertices are decided by
\eqref{eq:edge-rule} in both graphs. By \Cref{lem:comparison}(i), the
indegree processes of \(\mathbf z_r\) in the two graphs agree on
\([u_r,T_r)\). By \Cref{lem:parts}(iii), the older endpoint \(\mathbf z_r\)
of a required transition of the \(j\)th part contributes to this part. The
transition is decided at a time \(t\le T_r\), using the indegree of
\(\mathbf z_r\) at time \(t-\) and the mark in \(\Upsilon\). Hence \(A_j^L\)
occurs for the graph built from \(\Gamma_j\). We have shown that
\[
  \mathcal Q_S^{(0)}
  \cap
  \{\Gamma_1\cup\cdots\cup\Gamma_k\subseteq\Lambda_L\}
  \subseteq
  A_1^L\circ\cdots\circ A_k^L ,
\]
where \(\circ\) denotes disjoint occurrence. The BK inequality for marked
Poisson point processes on bounded domains of van den
Berg~\cite{VanDenBerg1996}, in the form used in
\cite[Section~6.2]{JacobMorters2017}, and the inclusion
\(A_j^L\subseteq A_j\) give
\[
\begin{aligned}
  \PP\Big(\mathcal Q_S^{(0)}\cap
  \{\Gamma_1\cup\cdots\cup\Gamma_k\subseteq\Lambda_L\}
  \,\Big|\,\Upsilon\Big)
  &\le
  \PP\big(A_1^L\circ\cdots\circ A_k^L\,\big|\,\Upsilon\big)\\
  &\le
  \prod_{j=1}^k\PP\big(A_j^L\,\big|\,\Upsilon\big)
  \le
  \prod_{j=1}^k\PP\big(A_j\,\big|\,\Upsilon\big).
\end{aligned}
\]
Letting \(L\to\infty\), we obtain
\(\PP(\mathcal Q_S^{(0)}\mid\Upsilon)\le\prod_{j=1}^k\PP(A_j\mid\Upsilon)\).

The event \(A_j\) depends on the indegree processes of the vertices
contributing to the \(j\)th part before the times \(T_r\). By
\Cref{lem:parts}(iv), and since the older path neighbours of a vertex are
never its children, these indegree processes depend on \(\Upsilon\) only
through the marks of the transitions of the \(j\)th part. The same holds for
the marks of the required transitions themselves. Since every transition
belongs to exactly one part, the random variables
\(\PP(A_j\mid\Upsilon)\), \(1\le j\le k\), are independent. Taking
expectations gives
\[
  \PP\big(\mathcal Q_S^{(0)}\big)
  \le
  \prod_{j=1}^k\PP(A_j).
\]
By \Cref{lem:part-bound} and \(\sum_{j=1}^k|R_j|=|R|\), this proves
\eqref{eq:quick-reduced}.
\end{proof}

\section{The separated backbone}\label{sec:backbone}

After constructing \(G^\infty\), add deterministic short edges by setting
\[
  \{\mathbf x,\mathbf y\}\in E(\whG)
  \quad\Longleftrightarrow\quad
  \{\mathbf x,\mathbf y\}\in E(G^\infty)
  \text{ or }q(\mathbf x,\mathbf y)\le1.
\]
The added edges do not affect any degree process.

\begin{lemma}\label{lem:short-local-finite}
The deterministic short-edge graph is locally finite. Under the root Palm law
with root \((0,U)\), \(U\) uniform on \((0,1]\), its expected short degree is
finite.
\end{lemma}

\begin{proof}
Fix \(\mathbf x=(x,s)\). For younger vertices \(t\ge s\), the short-edge
condition gives a ball of volume \(v_d s^{-\gamma}t^{\gamma-1}\), hence mean
\[
  v_d s^{-\gamma}\int_s^1t^{\gamma-1}\dd t
  =
  \frac{v_d}{\gamma}(s^{-\gamma}-1).
\]
For older vertices \(t<s\), the mean is
\[
  v_d s^{\gamma-1}\int_0^s t^{-\gamma}\dd t
  =
  \frac{v_d}{1-\gamma}.
\]
Both are finite for fixed \(s\), and the first expression is integrable when
\(s=U\), because \(\EE[U^{-\gamma}]<\infty\).

For a fixed deterministic \(\mathbf x\) the set of short neighbours has finite
Poisson intensity, and hence is finite almost surely under the Palm law at
\(\mathbf x\). By Mecke's formula, the expected number of Poisson vertices in
any unit spatial cube whose short degree is infinite is therefore zero.
Summing over the countable collection of integer translates gives local
finiteness of the short-edge graph almost surely. The expected Palm degree
statement follows from the same displayed formulae and the integrability in
\(U\).
\end{proof}

We now explain how the Jacob--M\"orters quick-path construction
and the skeleton reduction of Gracar, L\"uchtrath and
M\"orters~\cite{GracarLuechtrathMorters2021} are
combined.
The quick-path estimate is applied first, not to the final skeleton but to an
expanded finite segment of a retained backbone. In this segment, consecutive
retained backbone vertices may be joined by an actual edge, by a deterministic
short transition, or by inserting one actual common child. These inserted
children are then integrated out before skeletonising. The skeleton is
extracted only afterwards,
by deleting local birth-time maxima among the retained backbone vertices.
Thus the skeleton vertices form a subsequence of the retained backbone, while
the nonskeleton backbone vertices form connector blocks between consecutive
skeleton vertices. The separation enforced by the quickened backbone allows
us to remove every inserted child and every nonskeleton local maximum by the
same two-edge contraction.

\begin{proposition}\label{prop:backbone}
If a retained vertex belongs to an infinite component of \(\pwhG^\infty\),
then there is an infinite retained backbone
\[
  \mathbf y_0,\mathbf y_1,\ldots
\]
with the following properties.
\begin{enumerate}[label=(\roman*)]
\item Consecutive backbone vertices are connected either by an actual edge of
  \(G^\infty\), by a deterministic short edge of \(\whG\), or through one
  inserted actual common child.
\item Nonconsecutive backbone vertices are not adjacent in \(\whG\) and have
  no actual common child in \(G^\infty\). In particular,
  \(q(\mathbf y_i,\mathbf y_j)>1\) for \(j\ge i+2\).
\item After inserting the common children, the expanded sequence is an
  augmented quick path, with the deterministic short edges as short
  transitions, to which \Cref{prop:quick-path}
  applies.
\item For every finite backbone segment, first integrating the inserted common
  children with \Cref{lem:static-contraction} and then deleting local
  birth-time maxima of the retained backbone preserves the separation
  \(q>1\) at every deletion.
\end{enumerate}
\end{proposition}

\begin{proof}
The graph \(\pwhG^\infty\) is locally finite: actual degrees are finite
almost surely \cite[Section~4]{JacobMorters2017}, and short degrees are
finite by \Cref{lem:short-local-finite}. An infinite component of a locally
finite graph contains a geodesic ray \(\mathbf x_0,\mathbf x_1,\ldots\)
started at any of its vertices; this follows from a diagonal argument, cf.\
\cite[Section~6]{JacobMorters2017}. Geodesicity forbids an edge in \(\whG\)
between \(\mathbf x_i\) and \(\mathbf x_j\) whenever \(j\ge i+2\), and
therefore gives \(q(\mathbf x_i,\mathbf x_j)>1\).

We quicken this ray to remove nonconsecutive common children. Put
\(\psi(0)=0\). Given \(\psi(n)\), let
\[
  \cN_n=\{k>\psi(n):
  \mathbf x_{\psi(n)}\text{ and }\mathbf x_k
  \text{ have an actual common child in }G^\infty\}.
\]
This set is finite, because every such common child is an actual neighbour of
\(\mathbf x_{\psi(n)}\), and all actual degrees are finite. If
\(\cN_n\ne\varnothing\), set \(\psi(n+1)=\max\cN_n\); otherwise set
\(\psi(n+1)=\psi(n)+1\). Define \(\mathbf y_n=\mathbf x_{\psi(n)}\). Then
the sequence is infinite and retained. If \(j\ge i+2\), then
\(\psi(j)>\psi(i+1)\), so \(\mathbf y_i\) and \(\mathbf y_j\) cannot have a
common child by maximality of \(\psi(i+1)\). They also cannot be adjacent in
\(\whG\), since their indices on the original geodesic differ by at least two.

Consecutive backbone vertices are adjacent in \(\whG\) if and only if they
were consecutive on the geodesic; in the maximal-jump case they share an
actual common child, and both can occur simultaneously, namely when
\(\max\cN_n=\psi(n)+1\). Between every pair of consecutive backbone
vertices that are not adjacent in \(\whG\), and only between such pairs, we
insert the oldest common child. Such pairs arise only from a maximal jump
\(\psi(n+1)\ge\psi(n)+2\), so the required common child
exists.

The inserted children are distinct and are not backbone vertices. If an
inserted child coincided with a backbone vertex other than its two endpoints,
this backbone vertex would be adjacent in \(G^\infty\) to a nonconsecutive
backbone vertex. If it coincided with another inserted child, it would be a
common child of two nonconsecutive backbone vertices. Both contradict the
properties established above. Every inserted child is
younger than both adjacent backbone endpoints, hence is a local birth-time
maximum in the expanded sequence. The usual quick-path exclusions for
age-regular vertices follow from the absence of nonconsecutive augmented edges
and nonconsecutive common children. For (q2), if the intermediate
local maximum is an inserted child, it is the oldest actual common child by
construction. If it is a retained backbone vertex, its two neighbours are
nonconsecutive backbone vertices and have no actual common child. Thus
\Cref{prop:quick-path} applies.

For a finite segment, integrate inserted common children before
skeletonising. Whenever a child was inserted, its two backbone endpoints are
not adjacent in \(\whG\), hence separated by \(q>1\), so
\Cref{lem:static-contraction} contracts the two child kernels to one endpoint
kernel. After this first stage, only the retained backbone remains. Delete
interior local birth-time maxima, always choosing the largest birth time among
the current local maxima. If \(\mathbf y_k\) is deleted and its current
neighbours have original indices \(i<k<j\), then \(j-i\ge2\). The separation
of nonconsecutive backbone vertices gives \(q(\mathbf y_i,\mathbf y_j)>1\),
so \Cref{lem:static-contraction} applies at every deletion.
\end{proof}

\begin{corollary}\label{cor:reduced-backbone}
There is \(B_d<\infty\) such that the following holds under the retained root
Palm law. Fix a backbone with \(n\) transitions and suppose its skeleton has
\(m\) transitions, with block lengths
\[
  r_1+\cdots+r_m=n,\qquad r_i\ge1.
\]
For prescribed skeleton vertices \(\mathbf s_0,\ldots,\mathbf s_m\), after
summing transition labels and ordered binary-tree connector shapes and after
integrating inserted common children and nonskeleton retained backbone
vertices, the resulting density is bounded by
\[
  (B_dp)^n
  \prod_{j=1}^mK(\mathbf s_{j-1},\mathbf s_j).
\]
\end{corollary}

\begin{proof}
Let \(C_{\mathrm{con}}\ge1\) be a constant for
\Cref{lem:static-contraction}. Fix the transition labels and the ordered
binary-tree shape in every connector block. Suppose \(h\le n\) backbone
transitions are represented by an inserted common child. The expanded
augmented quick path has \(n+h\) transitions, so
\Cref{prop:quick-path} gives
\[
  C_{\mathrm{qp}}^{n+h}
  \prod K
\]
over the expanded sequence; deterministic short transitions contribute only
their indicators \(q\le1\), bounded by the corresponding \(K\)-factor.

Integrate the \(h\) inserted common children first, dropping the constraint
that such a child is the oldest common child; this only enlarges the witness
set. Their two adjacent backbone endpoints are separated by \(q>1\), so each
integration contributes a factor at most \(C_{\mathrm{con}}\) and leaves one
endpoint kernel.
After this step the kernel product lies on the retained backbone. Now fix the
skeleton. There are
\(n-m\) nonskeleton retained backbone vertices. For fixed ordered tree shapes,
successive local-maximum deletion contributes a factor at most
\(C_{\mathrm{con}}^{n-m}\) and leaves
\(\prod_{j=1}^mK(\mathbf s_{j-1},\mathbf s_j)\). By the encoding of connector blocks by ordered binary trees of
\cite[Section~2]{GracarLuechtrathMorters2021} and the Catalan bound, the
number of ordered binary-tree shapes in the connector blocks is bounded
by
\[
  \prod_{i=1}^m4^{r_i-1}=4^{n-m}.
\]
There are at most three labels per backbone transition, and the \(n\) non-root
backbone vertices are all retained, giving the full factor \(p^n\).
Therefore the contribution is bounded by
\[
  3^n4^{n-m}p^n
  C_{\mathrm{qp}}^{n+h}
  C_{\mathrm{con}}^{h+n-m}
  \prod_{j=1}^mK(\mathbf s_{j-1},\mathbf s_j)
  \le
  (12C_{\mathrm{qp}}^2C_{\mathrm{con}}^2p)^n
  \prod_{j=1}^mK(\mathbf s_{j-1},\mathbf s_j).
\]
Thus the claim holds with \(B_d=12C_{\mathrm{qp}}^2C_{\mathrm{con}}^2\).
\end{proof}

\section{Skeleton first moment}\label{sec:skeleton}

The diffuse birth-time coordinate implies that, almost surely, the skeleton
birth times have a unique oldest vertex. If this vertex has index \(k\), then
\[
  t_0>t_1>\cdots>t_k<t_{k+1}<\cdots<t_m .
\]
For \(L>0\), let \(\mathcal D_{m,k}(L)\) be the set of such birth-time
vectors with \(t_k\ge e^{-L}\), and set
\[
  T_{m,k}(L)
  :=
  \int_{\mathcal D_{m,k}(L)}
  \prod_{j=1}^m g(t_{j-1},t_j)^{-1}
  \dd t_0\cdots\dd t_m .
\]
For an interior minimum \(1\le k\le m-1\), the product telescopes to
\[
  \prod_{j=1}^m g(t_{j-1},t_j)^{-1}
  =
  t_0^{\gamma-1}t_m^{\gamma-1}t_k^{-2\gamma}
  \prod_{\substack{1\le j\le m-1\\j\ne k}}t_j^{-1}.
\]
At the endpoints,
\[
  \prod_{j=1}^m g(t_{j-1},t_j)^{-1}
  =
  \begin{cases}
  t_0^{-\gamma}t_m^{\gamma-1}
  \prod_{j=1}^{m-1}t_j^{-1},& k=0,\\[4pt]
  t_0^{\gamma-1}t_m^{-\gamma}
  \prod_{j=1}^{m-1}t_j^{-1},& k=m .
  \end{cases}
\]

We use the logarithmic substitution \(u_i=-\log t_i\), with
\(\dd t_i=e^{-u_i}\dd u_i\). For an interior minimum, all intermediate powers
cancel and
\[
  T_{m,k}(L)
  =
  \int_0^L e^{(2\gamma-1)U}A_k(U)A_{m-k}(U)\dd U,
\]
where
\[
  A_r(U)
  :=
  \int_{0<u_0<\cdots<u_{r-1}<U}e^{-\gamma u_0}
  \dd u_0\cdots\dd u_{r-1}
  \le
  \frac{U^{r-1}}{\gamma(r-1)!}.
\]
Indeed \(U=u_k\) is the logarithmic age of the oldest skeleton vertex, and
the left and right branches contribute the two \(A\)-factors. If the oldest
vertex is an endpoint, the skeleton is monotone, and direct integration gives
\[
  T_{m,0}(\infty)=T_{m,m}(\infty)=(1-\gamma)^{-m}.
\]

For a backbone of \(n\) transitions, call its skeleton regular at scale \(n\)
if its oldest birth time is at least \(2^{-n}\). Put
\[
  L_n=n\log2,
  \qquad
  T_m^{\mathrm{reg}}(n)=\sum_{k=0}^mT_{m,k}(L_n).
\]

\begin{proposition}\label{prop:time-bound}
For each \(\gamma\in(0,1)\), there are finite constants
\(C_\gamma,D_\gamma\) such that
\[
  T_m^{\mathrm{reg}}(n)\le C_\gamma D_\gamma^n,
  \qquad 1\le m\le n.
\]
\end{proposition}

\begin{proof}
Split the sum according to whether the oldest vertex is an endpoint or an
interior vertex of the skeleton.

For the two endpoints the skeleton is monotone, so
\(\mathcal D_{m,k}(L_n)\subseteq\mathcal D_{m,k}(\infty)\) and the values
computed above give
\[
  T_{m,0}(L_n)+T_{m,m}(L_n)
  \le
  T_{m,0}(\infty)+T_{m,m}(\infty)
  =
  2(1-\gamma)^{-m}
  \le
  2(1-\gamma)^{-n},
\]
since \(m\le n\) and \((1-\gamma)^{-1}>1\).

When \(m=1\), there are no interior minima. For \(m\ge2\) and an interior
minimum \(1\le k\le m-1\), insert
\[
  A_r(U)\le\frac{U^{r-1}}{\gamma(r-1)!}
\]
into the representation of \(T_{m,k}\), noting that
\((k-1)+(m-k-1)=m-2\):
\[
  T_{m,k}(L_n)
  \le
  \frac{1}{\gamma^2(k-1)!\,(m-k-1)!}
  \int_0^{L_n}e^{(2\gamma-1)U}U^{m-2}\dd U .
\]
Summing over \(k\) and using
\[
  \sum_{k=1}^{m-1}\frac{1}{(k-1)!(m-k-1)!}
  =
  \frac{2^{m-2}}{(m-2)!},
\]
we obtain
\[
  \sum_{k=1}^{m-1}T_{m,k}(L_n)
  \le
  \frac{2^{m-2}}{\gamma^2(m-2)!}
  \int_0^{L_n}e^{(2\gamma-1)U}U^{m-2}\dd U .
\]
Now use \(U^{m-2}/(m-2)!\le e^U\) and \(L_n=n\log2\). Then
\[
  \sum_{k=1}^{m-1}T_{m,k}(L_n)
  \le
  \frac{2^{m-2}}{\gamma^2}
  \int_0^{L_n}e^{2\gamma U}\dd U
  \le
  \frac{2^{m-2}}{2\gamma^3}2^{2\gamma n}
  \le
  \frac{1}{8\gamma^3}\left(2^{1+2\gamma}\right)^n .
\]
For \(m=1\), the interior sum is empty and satisfies the same bound.
Adding the endpoint contribution gives
\[
  T_m^{\mathrm{reg}}(n)
  \le
  \frac{1}{8\gamma^3}\left(2^{1+2\gamma}\right)^n
  +
  2(1-\gamma)^{-n}
  \le
  C_\gamma D_\gamma^n
\]
after choosing, for example,
\[
  D_\gamma=\max\{2^{1+2\gamma},(1-\gamma)^{-1}\},
  \qquad
  C_\gamma=\frac{1}{8\gamma^3}+2 .
\]
\end{proof}

\begin{proposition}\label{prop:no-backbone}
There is \(p_*>0\) such that, for every \(0<p<p_*\), the retained root Palm
law assigns probability zero to the event that the root belongs to an
infinite component of \(\pwhG^\infty\).
\end{proposition}

\begin{proof}
We call a sequence of retained vertices \(\mathbf y_0,\ldots,\mathbf y_n\),
with \(\mathbf y_0\) the retained Palm root, a \emph{separated backbone
witness} of length \(n\) if it has properties (i)--(iii) of
\Cref{prop:backbone}. Let \(N_n^{\mathrm{reg}}\) be the number of separated
backbone witnesses of length \(n\) whose skeleton is regular at
scale \(n\). For fixed skeleton length \(m\), combine the composition count
\(\binom{n-1}{m-1}\), \Cref{cor:reduced-backbone},
\Cref{lem:spatial-skeleton}, and \Cref{prop:time-bound}:
\[
\begin{aligned}
\EE[N_n^{\mathrm{reg}}]
&\le
(B_dp)^n
\sum_{m=1}^n \binom{n-1}{m-1}
J_{d,\eta}^mT_m^{\mathrm{reg}}(n)\\
&\le
C_\gamma
(B_dpD_\gamma)^n
J_{d,\eta}(1+J_{d,\eta})^{n-1}\\
&\le
C_\gamma\big[B_dpD_\gamma(1+J_{d,\eta})\big]^n .
\end{aligned}
\]

Let \(N_n^{\mathrm{end}}\) be the number of separated backbone witnesses of
length \(n\) whose endpoint is the oldest backbone vertex and whose
endpoint birth time is smaller than \(2^{-n}\). Here the skeleton is monotone,
so the endpoint estimate gives
\[
\begin{aligned}
  \EE[N_n^{\mathrm{end}}]
  \le
  (B_dp)^n
  \sum_{m=1}^n\binom{n-1}{m-1}
  \left(\frac{J_{d,\eta}}{1-\gamma}\right)^m
  \le
  \left[
    B_dp\left(1+\frac{J_{d,\eta}}{1-\gamma}\right)
  \right]^n .
\end{aligned}
\]
This estimate drops the constraint that the endpoint birth time is smaller
than \(2^{-n}\), which only increases the expectation. The constraint is used
only in the dichotomy below.
Choose
\[
  p_*=
  \frac{1}{2B_d\max\left\{
  D_\gamma(1+J_{d,\eta}),
  1+J_{d,\eta}/(1-\gamma)
  \right\}} .
\]
Then
\[
  \sum_{n\ge1}
  \PP\{N_n^{\mathrm{reg}}>0\}
  +
  \sum_{n\ge1}
  \PP\{N_n^{\mathrm{end}}>0\}
  <\infty .
\]

An infinite backbone has the following deterministic dichotomy. Write
\(\mathbf y_i=(y_i,t_i)\) and
\[
  a_n:=\min_{0\le i\le n}t_i .
\]
If \(a_n\ge2^{-n}\) for infinitely many \(n\), then the regular events
\(\{N_n^{\mathrm{reg}}>0\}\) occur infinitely often. Indeed, the oldest
vertex of the length-\(n\) prefix is not deleted in the skeletonisation,
because the global birth-time minimum is never a local birth-time maximum.
The skeleton of the prefix therefore has oldest birth time \(a_n\ge2^{-n}\)
and is regular at scale \(n\).
Otherwise
\(a_n<2^{-n}\) for all large \(n\).
Let \(j(n)\) be the index attaining the minimum in the first \(n\) steps.
The sequence \(j(n)\) is nondecreasing, since it changes only when a
strictly older vertex appears, and it is unbounded, since no fixed positive
birth time can be smaller than \(2^{-n}\) for arbitrarily large
\(n\). Along
a subsequence on which \(j(n)\) strictly increases, the prefix ending at
\(j(n)\) has its endpoint as oldest vertex and
\[
  t_{j(n)}=a_n<2^{-n}\le2^{-j(n)}.
\]
Thus the endpoint-old events
\(\{N_{j(n)}^{\mathrm{end}}>0\}\) occur infinitely often. The first
Borel--Cantelli lemma excludes both alternatives. By \Cref{prop:backbone}, an
infinite component of \(\pwhG^\infty\) would produce such a backbone.
\end{proof}

\begin{corollary}\label{cor:infinite-nonperc}
For every \(0<p<p_*\), the augmented graph \(\pwhG^\infty\) has no infinite
component almost surely. Consequently \(\Gp^\infty\) has no infinite component
almost surely.
\end{corollary}

\begin{proof}
\Cref{prop:no-backbone} says that the retained Palm root has probability zero
of lying in an infinite component of \(\pwhG^\infty\). By Campbell's formula,
for the unit cube \(Q=[0,1]^d\), if \(N_\infty(Q)\) denotes the number of
retained vertices in \(Q\times(0,1]\) lying in infinite components, then
\[
  \EE N_\infty(Q)
  =
  p\int_0^1
  \PP^p_{(0,u)}
  \{(0,u)\leftrightarrow\infty\text{ in }\pwhG^\infty\}\dd u
  =0.
\]
Here \(\PP^p_{(0,u)}\) is the Palm law with \((0,u)\) present and retained.
Stationarity and countability of integer translates imply that
\(\pwhG^\infty\) has no infinite component almost surely. Since
\(\Gp^\infty\subset\pwhG^\infty\), the original percolated graph also has no
infinite component.
\end{proof}

\section{Asymptotically linear attachment rules}\label{sec:general}

We now return to the general attachment rule of Theorem A. Choose
\[
  \gamma<\gamma_+<1,
  \qquad
  \delta\left(\frac1{\gamma_+}-1\right)>1,
\]
and then
\[
  \max\left\{1,\frac{\gamma_+}{1-\gamma_+}\right\}<\eta<\delta.
\]
Set \(\kappa_+:=\eta(1-\gamma_+)-\gamma_+>0\).
Since \(f(k)/k\to\gamma\), there is \(\beta_+<\infty\) such that
\[
  f(k)\le f_+(k):=\beta_++\gamma_+k,\qquad k\ge0.
\]

Since \(f_+(0)\ge f(0)>0\), the rule \(f_+\) is affine with slope
\(\gamma_+\) and \(a_+:=\beta_+/\gamma_+>0\). By the choice of \(\gamma_+\)
and \(\eta\), \Cref{sec:spatial,sec:dynamic,sec:backbone,sec:skeleton} apply to
\(f_+\), with \(\gamma_+\), \(a_+\) and \(\kappa_+\) in place of \(\gamma\),
\(a\) and \(\kappa\), and with \(q\) and \(K\) defined by \eqref{eq:g-q-K}
with \(\gamma_+\) in place of \(\gamma\).

Let \(G_+^\infty\) be the graph built by the rule \eqref{eq:edge-rule}, with
\(f_+\) in place of \(f\), from the same Poisson process \(\cX\) and the same
marks as \(G^\infty\). We show that \(G^\infty\) is a subgraph of
\(G_+^\infty\), using the argument of \Cref{lem:comparison}(ii).

\begin{lemma}\label{lem:monotone-coupling}
Almost surely, \(G^\infty\) is a subgraph of \(G_+^\infty\).
\end{lemma}

\begin{proof}
The two graphs have the same vertex set. Under the Palm law of a vertex
\(\mathbf x=(x,s)\), as at the beginning of \Cref{sec:dynamic}, the indegree
process of \(\mathbf x\) in \(G^\infty\) is, in logarithmic time
\(\log(v/s)\), a pure-birth process whose jump rate in state \(k\) is
\(f(k)\le\gamma_+(k+a_+)\). A pure-birth process with linearly bounded jump
rates does not explode, so \(\mathbf x\) has finitely many children in
\(G^\infty\) almost surely. By Mecke's formula, almost surely every vertex
has finitely many children in \(G^\infty\).

Fix a vertex \(\mathbf x\), and enumerate its finitely many children in
\(G^\infty\) as \(\mathbf y_1,\ldots,\mathbf y_m\), in increasing order of
birth time \(t_1<\cdots<t_m\). Write \(Z_{\mathbf x}\) and
\(Z_{\mathbf x}^+\) for its indegree processes in \(G^\infty\) and
\(G_+^\infty\), respectively. We show inductively that each \(\mathbf y_j\)
is also a child of \(\mathbf x\) in \(G_+^\infty\). If this holds for the
first \(j-1\) children, then
\[
  Z_{\mathbf x}(t_j-)=j-1\le Z_{\mathbf x}^+(t_j-).
\]
Since \(f\le f_+\) and \(f_+\) is nondecreasing,
\[
  f\big(Z_{\mathbf x}(t_j-)\big)
  \le
  f_+\big(Z_{\mathbf x}(t_j-)\big)
  \le
  f_+\big(Z_{\mathbf x}^+(t_j-)\big).
\]
Since \(\varphi\) is nonincreasing, the right side of \eqref{eq:edge-rule} for
the pair \(\mathbf x,\mathbf y_j\) is at least as large in \(G_+^\infty\) as in
\(G^\infty\). Both graphs use the same mark \(V_{\mathbf x,\mathbf y_j}\),
so \(\mathbf y_j\) is a child of \(\mathbf x\) in \(G_+^\infty\) as well.
Since \(\mathbf x\) was arbitrary, every edge of \(G^\infty\), which joins a
vertex to one of its children, is an edge of \(G_+^\infty\).

\end{proof}

We retain the same vertices in both graphs, and write \(\Gp_+^\infty\) for
the graph induced by the retained vertices in \(G_+^\infty\). By
\Cref{cor:infinite-nonperc}, applied to the affine rule \(f_+\), there is
\(p_*>0\) such that, for every \(0<p<p_*\), the graph \(\Gp_+^\infty\) has no
infinite component almost surely. By \Cref{lem:monotone-coupling},
\(\Gp^\infty\) is a subgraph of \(\Gp_+^\infty\). Hence, for every
\(0<p<p_*\),
\begin{equation}\label{eq:limit-nonperc}
  \Gp^\infty
  \quad\text{has no infinite component almost surely.}
\end{equation}

\section{From the limit to finite graphs}\label{sec:finite}

The local convergence of \(G^t\) to \(G^\infty\) is part of the
Jacob--M\"orters construction of spatial preferential attachment; see
\cite[Section~4]{JacobMorters2017} and \cite{JacobMorters2015}. The upper
bound needed here is the case \(\zeta=0\) of the upper bound on the giant in
\cite[Corollary~2.27]{VanDerHofstad2024}. Local convergence in probability
bounds the proportion of vertices in the largest component by the survival
probability of the local limit; see also
\cite[Corollary~2.1]{VanDerHofstad2021}.
Since Bernoulli percolation acts by independent thinning of the vertices,
local convergence in probability of \(G^t\) transfers directly to
\(\Gp^t\), with the retained Palm version of \(\Gp^\infty\) as the
limit.
For the reader's convenience, we give the Poissonised version of this
statement with a proof.

\begin{lemma}\label{lem:vdh-transfer}
Suppose \(\Gp^t\) converges locally in probability to the retained Palm
version of \(\Gp^\infty\), and suppose the latter has no infinite component
almost surely. Then
\[
  |C_1(\Gp^t)|/t\longrightarrow0
  \qquad\text{in probability.}
\]
\end{lemma}

\begin{proof}
For \(k\ge1\), define
\[
  \xi^{(k)}(\mathbf x,G)
  :=\1\{|\cC_G(\mathbf x)|\le k\}.
\]
This bounded functional is determined by the graph ball of radius \(k\). Local
convergence in probability, in the empirical-root form used in the upper
direction of the almost-local giant theorem of \cite{VanDerHofstad2021}, gives
\[
  \frac1t
  \sum_{\mathbf x\in\Gp^t}
  \xi^{(k)}(\mathbf x,\Gp^t)
  \xrightarrow{\PP}
  p\int_0^1
  \PP^p_{(0,u)}
  \{|\cC_{\Gp^\infty}(0,u)|\le k\}\dd u .
\]
Since the limit graph does not percolate, the right side increases to \(p\)
as \(k\to\infty\). Also \(|\Gp^t|/t\to p\) in probability by the law of
large numbers for the retained Poisson vertex set. Thus, for every
\(\varepsilon>0\), one can choose \(k\) so large that
\[
  \limsup_{t\to\infty}
  \PP\left\{
  \frac1t\#\{\mathbf x\in\Gp^t:
  |\cC_{\Gp^t}(\mathbf x)|>k\}>\varepsilon
  \right\}=0.
\]
If \(|C_1(\Gp^t)|>\lambda t\) and \(t>k/\lambda\), then the set counted
above has size at least \(\lambda t\). Taking \(\varepsilon<\lambda\) proves
the claim.
\end{proof}

\begin{proof}[Proof of Theorem A]
For item (i), we invoke the robust finite-graph theorem of
Jacob--M\"orters~\cite[Proposition~14]{JacobMorters2017}, in the
\(d\)-dimensional convention recalled above. For item (ii), choose the slack
parameters \(\gamma_+\) and \(\eta\) as in \Cref{sec:general}.
\Cref{lem:monotone-coupling} and \Cref{cor:infinite-nonperc}, applied to the
affine rule \(f_+\), give \eqref{eq:limit-nonperc} for every \(0<p<p_*\). Applying
\Cref{lem:vdh-transfer}, with the local convergence provided by
Jacob--M\"orters, gives the finite-graph conclusion.
\end{proof}

\begin{remark}[Critical line]
The proof requires an envelope exponent \(1<\eta<\delta\) with
\(\eta(1/\gamma-1)>1\). On the critical line
\(\delta(1/\gamma-1)=1\), no such exponent exists. At the formal endpoint
\(\eta=\delta\), the time integral in the static two-edge contraction
of \Cref{lem:static-contraction} is logarithmic.
\end{remark}

\section*{Acknowledgements}
We thank Lukas L\"uchtrath for helpful discussions about this work, in
particular for the interpretation of the product exponent in
\Cref{rem:product-exponent}.

\printbibliography

@article{JacobMorters2015,
  author       = {Jacob, Emmanuel and M{\"o}rters, Peter},
  title        = {Spatial preferential attachment networks: power laws and clustering coefficients},
  journaltitle = {The Annals of Applied Probability},
  shortjournal = {Ann. Appl. Probab.},
  volume       = {25},
  number       = {2},
  pages        = {632--662},
  date         = {2015},
  issn         = {1050-5164},
  doi          = {10.1214/14-AAP1006},
  eprint       = {1210.3830},
  eprinttype   = {arxiv},
  eprintclass  = {math.PR},
  url          = {https://zbmath.org/6430707},
  addendum     = {Zbl~1310.05185}
}

@article{JacobMorters2017,
  author       = {Jacob, Emmanuel and M{\"o}rters, Peter},
  title        = {Robustness of scale-free spatial networks},
  journaltitle = {The Annals of Probability},
  shortjournal = {Ann. Probab.},
  volume       = {45},
  number       = {3},
  pages        = {1680--1722},
  date         = {2017},
  issn         = {0091-1798},
  doi          = {10.1214/16-AOP1098},
  eprint       = {1504.00618},
  eprinttype   = {arxiv},
  eprintclass  = {math.PR},
  url          = {https://zbmath.org/6754783},
  addendum     = {Zbl~1367.05194}
}

@article{GracarLuechtrathMorters2021,
  author       = {Gracar, Peter and L{\"u}chtrath, Lukas and M{\"o}rters, Peter},
  title        = {Percolation phase transition in weight-dependent random connection models},
  journaltitle = {Advances in Applied Probability},
  shortjournal = {Adv. Appl. Probab.},
  volume       = {53},
  number       = {4},
  pages        = {1090--1114},
  date         = {2021},
  issn         = {0001-8678},
  doi          = {10.1017/apr.2021.13},
  eprint       = {2003.04040},
  eprinttype   = {arxiv},
  eprintclass  = {math.PR},
  url          = {https://zbmath.org/7435082},
  addendum     = {Zbl~1475.60198}
}

@article{DereichMorters2013,
  author       = {Dereich, Steffen and M{\"o}rters, Peter},
  title        = {Random networks with sublinear preferential attachment: the giant component},
  journaltitle = {The Annals of Probability},
  shortjournal = {Ann. Probab.},
  volume       = {41},
  number       = {1},
  pages        = {329--384},
  date         = {2013},
  issn         = {0091-1798},
  doi          = {10.1214/11-AOP697},
  eprint       = {1007.0899},
  eprinttype   = {arxiv},
  eprintclass  = {math.PR},
  url          = {https://zbmath.org/6145506},
  addendum     = {Zbl~1260.05143}
}

@article{EckhoffMortersOrtgiese2018,
  author       = {Eckhoff, Maren and M{\"o}rters, Peter and Ortgiese, Marcel},
  title        = {Near critical preferential attachment networks have small giant components},
  journaltitle = {Journal of Statistical Physics},
  shortjournal = {J. Stat. Phys.},
  volume       = {173},
  number       = {3--4},
  pages        = {663--703},
  date         = {2018},
  issn         = {0022-4715},
  doi          = {10.1007/s10955-018-2054-5},
  eprint       = {1712.00400},
  eprinttype   = {arxiv},
  eprintclass  = {math.PR},
  url          = {https://zbmath.org/6996593},
  addendum     = {Zbl~1404.05192}
}

@article{MortersSchleicher2026,
  author       = {M{\"o}rters, Peter and Schleicher, Nick},
  title        = {The largest subcritical component in inhomogeneous random graphs of preferential attachment type},
  journaltitle = {Combinatorics, Probability and Computing},
  shortjournal = {Comb. Probab. Comput.},
  volume       = {35},
  number       = {3},
  pages        = {414--432},
  date         = {2026},
  issn         = {0963-5483},
  doi          = {10.1017/S096354832610039X},
  eprint       = {2503.05469},
  eprinttype   = {arxiv},
  eprintclass  = {math.PR},
  url          = {https://zbmath.org/8230664}
}

@article{HirschMoench2020,
  author       = {Hirsch, Christian and M{\"o}nch, Christian},
  title        = {Distances and large deviations in the spatial preferential attachment model},
  journaltitle = {Bernoulli},
  shortjournal = {Bernoulli},
  volume       = {26},
  number       = {2},
  pages        = {927--947},
  date         = {2020},
  issn         = {1350-7265},
  doi          = {10.3150/19-BEJ1121},
  eprint       = {1809.09956},
  eprinttype   = {arxiv},
  eprintclass  = {math.PR},
  url          = {https://zbmath.org/7166552},
  addendum     = {Zbl~1466.60209}
}

@article{GracarGrauerLuechtrathMorters2019,
  author       = {Gracar, Peter and Grauer, Arne and L{\"u}chtrath, Lukas and M{\"o}rters, Peter},
  title        = {The age-dependent random connection model},
  journaltitle = {Queueing Systems},
  shortjournal = {Queueing Syst.},
  volume       = {93},
  number       = {3--4},
  pages        = {309--331},
  date         = {2019},
  issn         = {0257-0130},
  doi          = {10.1007/s11134-019-09625-y},
  eprint       = {1810.03429},
  eprinttype   = {arxiv},
  eprintclass  = {math.PR},
  url          = {https://zbmath.org/7170165},
  addendum     = {Zbl~1432.60090}
}

@article{GracarHeydenreichMoenchMorters2022,
  author       = {Gracar, Peter and Heydenreich, Markus and M{\"o}nch, Christian and M{\"o}rters, Peter},
  title        = {Recurrence versus transience for weight-dependent random connection models},
  journaltitle = {Electronic Journal of Probability},
  shortjournal = {Electron. J. Probab.},
  volume       = {27},
  eid          = {60},
  pages        = {1--31},
  date         = {2022},
  issn         = {1083-6489},
  doi          = {10.1214/22-EJP748},
  eprint       = {1911.04350},
  eprinttype   = {arxiv},
  eprintclass  = {math.PR},
  url          = {https://zbmath.org/7541196},
  addendum     = {Zbl~1491.05168}
}

@article{GracarGrauerMorters2022,
  author       = {Gracar, Peter and Grauer, Arne and M{\"o}rters, Peter},
  title        = {Chemical distance in geometric random graphs with long edges and scale-free degree distribution},
  journaltitle = {Communications in Mathematical Physics},
  shortjournal = {Commun. Math. Phys.},
  volume       = {395},
  number       = {2},
  pages        = {859--906},
  date         = {2022},
  issn         = {0010-3616},
  doi          = {10.1007/s00220-022-04445-3},
  eprint       = {2108.11252},
  eprinttype   = {arxiv},
  eprintclass  = {math.PR},
  url          = {https://zbmath.org/7593345},
  addendum     = {Zbl~1497.05244}
}

@article{Moench2024,
  author       = {M{\"o}nch, Christian},
  title        = {Inhomogeneous long-range percolation in the weak decay regime},
  journaltitle = {Probability Theory and Related Fields},
  shortjournal = {Probab. Theory Relat. Fields},
  volume       = {189},
  number       = {3--4},
  pages        = {1129--1160},
  date         = {2024},
  issn         = {0178-8051},
  doi          = {10.1007/s00440-024-01281-5},
  eprint       = {2303.02027},
  eprinttype   = {arxiv},
  eprintclass  = {math.PR},
  url          = {https://zbmath.org/7891480},
  addendum     = {Zbl~1552.60284}
}

@article{GracarLuechtrathMoench2025,
  author       = {Gracar, Peter and L{\"u}chtrath, Lukas and M{\"o}nch, Christian},
  title        = {Finiteness of the percolation threshold for inhomogeneous long-range models in one dimension},
  journaltitle = {Electronic Journal of Probability},
  shortjournal = {Electron. J. Probab.},
  volume       = {30},
  eid          = {134},
  pages        = {1--29},
  date         = {2025},
  issn         = {1083-6489},
  doi          = {10.1214/25-EJP1399},
  eprint       = {2203.11966},
  eprinttype   = {arxiv},
  eprintclass  = {math.PR},
  url          = {https://zbmath.org/8102212},
  addendum     = {Zbl~1573.05340}
}

@article{JacobJahnelLuechtrath2026,
  author       = {Jacob, Emmanuel and Jahnel, Benedikt and L{\"u}chtrath, Lukas},
  title        = {Subcritical annulus crossing in spatial random graphs},
  journaltitle = {Advances in Applied Probability},
  shortjournal = {Adv. Appl. Probab.},
  pages        = {1--40},
  date         = {2026},
  issn         = {0001-8678},
  note         = {First View},
  doi          = {10.1017/apr.2026.10063},
  eprint       = {2411.10333},
  eprinttype   = {arxiv},
  eprintclass  = {math.PR},
  url          = {https://zbmath.org/902606513}
}

@article{VanDenBerg1996,
  author       = {van den Berg, J.},
  title        = {A note on disjoint-occurrence inequalities for marked Poisson point processes},
  journaltitle = {Journal of Applied Probability},
  shortjournal = {J. Appl. Probab.},
  volume       = {33},
  number       = {2},
  pages        = {420--426},
  date         = {1996},
  issn         = {0021-9002},
  doi          = {10.2307/3215065},
  url          = {https://zbmath.org/926245},
  addendum     = {Zbl~0860.60013}
}

@book{BinghamGoldieTeugels1987,
  author       = {Bingham, N. H. and Goldie, C. M. and Teugels, Jozef L.},
  title        = {Regular variation},
  series       = {Encyclopedia of Mathematics and Its Applications},
  number       = {27},
  publisher    = {Cambridge University Press},
  location     = {Cambridge},
  date         = {1987},
  pagetotal    = {491},
  isbn         = {0-521-30787-2},
  url          = {https://zbmath.org/4000257},
  addendum     = {Zbl~0617.26001}
}

@misc{VanDerHofstad2021,
  author       = {van der Hofstad, Remco},
  title        = {The giant in random graphs is almost local},
  date         = {2025},
  origdate     = {2021},
  version      = {3},
  note         = {Preprint},
  doi          = {10.48550/arXiv.2103.11733},
  eprint       = {2103.11733v3},
  eprinttype   = {arxiv},
  eprintclass  = {math.PR},
  url          = {https://zbmath.org/900561687}
}

@book{VanDerHofstad2024,
  author       = {van der Hofstad, Remco},
  title        = {Random graphs and complex networks. Volume 2},
  series       = {Cambridge Series in Statistical and Probabilistic Mathematics},
  number       = {54},
  publisher    = {Cambridge University Press},
  location     = {Cambridge},
  date         = {2024},
  pagetotal    = {490},
  doi          = {10.1017/9781316795552},
  isbn         = {978-1-107-17400-9},
  url          = {https://zbmath.org/7782966},
  addendum     = {Zbl~1540.05003}
}

\end{document}